\documentclass[a4paper,12pt,fleqn]{article}
\usepackage{graphicx}

\usepackage{mathtools}
\usepackage{mathrsfs}
\usepackage{enumerate,enumitem}
\usepackage[hypertexnames=false,hidelinks]{hyperref}
\usepackage{cleveref}
\usepackage{color}
\usepackage[utf8]{inputenc}
\usepackage[square,numbers]{natbib}
\usepackage{enumerate,enumitem}
\usepackage{amsmath,amssymb,amsthm,amstext}
\usepackage{mathrsfs}
\usepackage{color,graphicx}
\usepackage{comment}
\usepackage{fancyhdr}
\usepackage{vwcol,array}
\usepackage{emptypage}
\usepackage{footnote}
\makesavenoteenv{tabular}
\usepackage{mathtools}
\usepackage{verbatim}
\usepackage[hypertexnames=false]{hyperref}
\usepackage{cleveref}
\newtheorem{theorem}{Theorem}
\newtheorem{corollary}[theorem]{Corollary}
\newtheorem{lemma}[theorem]{Lemma}
\newtheorem{proposition}[theorem]{Proposition}
\theoremstyle{definition}
\newtheorem{definition}[theorem]{Definition}
\newtheorem{remark}[theorem]{Remark}

\newtheorem*{mainass}{Assumptions}
\newtheorem*{definition*}{Definition}
\newtheorem{example}[theorem]{Example}

\def\dd{\hspace*{1pt}\mathrm{d}}
\def\ee{\mathrm e}
\def\ii{\mathrm i}

\def\NN{\mathbb N}
\def\RR{\mathbb R}

\def\mb{\mathbb}
\def\mr{\mathrm}
\def\mc{\mathcal}
\def\ms{\mathscr}
\def\mk{\mathfrak}
\def\mf{\mathbf}

\def\Ell{\ms L}

\def\C{\mathrm C}

\def\AC{(\mathrm{ACP})}
\def\NC{(\mathrm{NCP})}

\def\sup{\operatorname{sup}}
\def\max{\operatorname{max}}

\def\U{\hspace*{1pt}\mathcal U}
\def\Q{\widetilde Q}
\def\P{\widetilde P}
\def\dt{\tfrac{\dd}{\dd t}}
\def\ds{\tfrac{\dd}{\dd s}}
\def\ddt{\frac{\dd}{\dd t}}

\def\t{\prescript{}{t}}

\def\veps{\varepsilon}

\def\Id{\operatorname{Id}}

\renewcommand\Re{\mathfrak R}

\renewcommand\max{\operatornamewithlimits{max}}

\renewcommand\sup{\operatornamewithlimits{sup}}

\newenvironment{enumI}{%
  \begin{enumerate}[leftmargin=*,align=left,itemsep=0cm,itemindent=0cm,label=(\roman*), ref=\roman*]
}{%
  \end{enumerate}
}

\newenvironment{enumA}{%
  \begin{enumerate}[leftmargin=*,align=left,itemsep=0cm,itemindent=0cm,label=(\alph*), ref=\alph*]
}{%
  \end{enumerate}
}  

\makeatletter
\newcommand{\nhyt}[2]{%
  \hypertarget{#1}{#2}%
    \protected@write\@mainaux{}{%
        \string\expandafter\string\gdef
          \string\csname\string\detokenize{#1}\string\endcsname{#2}%
    }%
  }
\newcommand{\hyl}[1]{%
  \hyperlink{#1}{\csname #1\endcsname}%
  }
\makeatother

\newcommand{\hyt}[2]{\nhyt{#2}{#1}}

\AtBeginDocument{
  \label{CorrectFirstPageLabel}
  
}

\begin{document}

\title{A second-order Magnus-type integrator for evolution equations with delay}

\author{Petra Csomós\thanks{Alfréd Rényi Institute of Mathematics, Reáltanoda utca 13-15., 1053 Budapest, Hungary,
ELTE Eötvös Loránd University and MTA-ELTE Numerical Analysis and Large Networks Research Group, Pázmány Péter sétány 1/C, 1117 Budapest, Hungary}
and Dávid Kunszenti-Kovács\thanks{Alfréd Rényi Institute of Mathematics, Reáltanoda utca 13-15., 1053 Budapest, Hungary}}
\date{\today}

\maketitle

\begin{abstract}

{
We rewrite abstract delay equations as nonautonomous abstract Cauchy problems allowing us to introduce a Magnus-type integrator for the former. We prove the second-order convergence of the obtained Magnus-type integrator. We also show that if the differential operators involved admit a common invariant set for their generated semigroups, then the Magnus-type integrator will respect this invariant set as well, allowing for much weaker assumptions to obtain the desired convergence. As an illustrative example we consider a space-dependent epidemic model with latent period and diffusion.
}

{Magnus integrator; quasilinear delay equation; convergence analysis; invariant sets; delayed epidemic model with space-dependence and diffusion.}
\end{abstract}

\section{Introduction}

The aim of this paper is to adapt the Magnus integrator for nonautonomous homogeneous problems to a wide class of nonautonomous delay problems. We are interested in problems of the form
\begin{equation}\label{eq:delay}\tag{QDE$_\varphi$}
\left\{\begin{aligned}
\dt u(t) &= Q(F(\t u))u(t), & t\in[0,\infty), \\
u(s) &= \varphi(s)\in X, & s\in[-\delta,0],
\end{aligned}
\right.
\end{equation}
where $X$ is a Banach space, $\t u\colon [-\delta,0]\to X$ denotes the $\delta$-history $\t u(s):=u(t+s)$ of the solution at time $t$ (i.e., the initial condition could also be written as $\prescript{}{0}u=\varphi$ with some initial history function $\varphi\colon [-\delta,0]\to X$), $Q(w)=Q_0+\Q(w)$ where $Q_0$ is an unbounded operator on $X$ and $\Q(w)$ is bounded for all $w\in X$, and $F\in \C([-\delta,0],X)$ actually only depends on the restriction to $[-\delta,-\epsilon]$ for some fixed $\epsilon\in (0,\delta]$. The role of $\epsilon$ and of the assumption that $F$ only depends on the restriction to $[-\delta,-\varepsilon]$ is to ensure that equation \eqref{eq:delay} never becomes implicit in the sense that the operator $Q(F(\t u))$ is determined already by time $t-\epsilon$ for all $t\ge 0$. The exact assumptions on the operators and functions involved will be detailed later.\\
Note that the case $Q_0=0$ corresponds to a quasilinear delay equation with bounded operators, while $Q_0\ne 0$ leads to the unbounded case. 
\medskip

In \citet{Magnus}, Magnus set out to solve the nonautonomous homogeneous problem
$\dt Y(t) = A(t)Y(t)$ ($t\geq0$), $Y(0)=Y_0$, where $Y(t)$ and $A(t)$ are linear operators on appropriate spaces. In the case when all of the operators $A(t)$ commute, the solution takes the simple form
$Y(t) = \exp \left( \int_{0}^t A(s)\dd s \right) Y_0$. In the general, noncommutative case, however, one has to correct the exponent, and the exact solution is given by the Magnus series expansion, involving integrals of commutators. The first term of this expansion corresponds to the commutative solution, and is a good approximant leading to a second-order numerical method using the midpoint rule to approximate the integral in the exponent:
\begin{equation}\label{eq:magnusY}
\widehat Y_{n+1}^{(\tau)} = \ee^{\tau A((n+1/2)\tau)}\widehat Y_n^{(\tau)}, \quad n\in\NN
\end{equation}
where $\tau>0$ is an arbitrary timestep and $\widehat Y_n^{(\tau)}$ denotes the numerical approximation to $Y(n\tau)$ with initial value $Y_0^{(\tau)}:=Y_0$.\\
Convergence of the classical Magnus integrators, such as \eqref{eq:magnusY}, has been widely studied in the literature for nonautonomous problems without delay. For finite dimensional spaces, the Magnus expansion, being the basis of the Magnus integrators, were analysed in \citet{Blanes-etal1} and \citet{Blanes-etal2}. The authors gave a condition on the Magnus expansion's convergence in \citet{Moan-Niesen}. The expansion in case of nonlinear equations was investigated in \citet{Casas-Iserles}. In \citet{Csomos}, a Magnus-type integrator (see \eqref{eq:magnus0} below) was derived for delay equations, moreover, its second-order convergence and positivity preserving property were also shown.\\
For infinite dimensional undelayed problems with inhomogeneity, in \citet{Gonzalez-etal} a numerical method was derived by using the Magnus expansion, and its second-order convergence was shown for sufficiently smooth solutions, in addition to (lower order) error bounds measured in the domain. The authors proved the first-order (operator) norm convergence of the same method under very mild conditions for homogeneous problems in \citet{Batkai-Sikolya}. \medskip

Delay equations describe processes where the time evolution of the unknown function not only depends on the actual state of the system but also on its past values. Delay differential equations find more and more frequent application in the modelling of scientific, financial, or even social phenomena, since there a delay term often naturally appears among the processes. One may here think of the role of the latent period when modelling the spread of epidemic diseases, the pregnancy period in population models, or the reaction time in any social or financial model.
In our problem \eqref{eq:delay}, the time parameters $\delta,\epsilon>0$ delimit the time window into the past that is influencing the present dynamics in the problem.
\medskip

Since the exact solution of problem \eqref{eq:delay} is difficult or even impossible to compute analytically, one needs to find a way to approximate it.
To this end, we first reformulate the delay problem as a nonautonomous problem, and then present a novel Magnus-type integrator based on Magnus integrators introduced for nonautonomous problems.
\medskip

For the reformulation as a nonautonomous problem, we first define 
the function $\widetilde u\colon[-\delta,\infty)\to X$ as
\begin{equation*}
\widetilde u(t):=\begin{cases}
\varphi(t), & t\in[-\delta,0], \\
u(t), & t\in[0,\infty),
\end{cases}
\end{equation*}
and the operators
\begin{equation*}
A(t):=Q(F(\t{\widetilde u})) \text{ for all } t\ge 0.
\end{equation*}
Then the function $\widetilde u$ satisfies the problem
\begin{equation}
\label{eq:ncp}
\left\{
\begin{aligned}
\dt \widetilde u(t) &= A(t)\widetilde u(t), & t\in[0,\infty), \\
\widetilde u(0) &= \varphi(0).
\end{aligned}
\right.
\end{equation}
Although the operators $A(t)$ depend on the unknown function $\widetilde u$ via $\t{\widetilde u}$, the minimum delay $\epsilon$ within the $Q$ term (stemming from $F$) allows problem \eqref{eq:ncp} to be treated as a nonautonomous Cauchy problem a posteriori. Indeed, for $t\in[0,\epsilon]$, $A(t)$ is actually determined as it only depends on the known restriction $\t{\widetilde u}|_{[-\delta,-\epsilon]}=\varphi|_{[t-\delta,t-\epsilon]}$. Hence solving the problem iteratively on the sub-intervals $[j\epsilon,(j+1)\epsilon]$, $j=0,1,\ldots$, yields an explicit nonautonomous Cauchy problem for each time-segment considered, provided well-posedness is maintained along the way. Hence, we will consider problem \eqref{eq:ncp} essentially as a formally nonautonomous problem on the whole time interval. The solution $\widetilde u$ can then be approximated by the method derived by Magnus in \citet{Magnus} for such problems. This method however uses the values $F(\t u)$, which have to be themselves approximated, involving an appropriate second-order approximation of $F$ based off of a grid that is compatible with the one used for $u$. \medskip

To this end, one defines an arbitrary $N\in\NN$, takes the time step $\tau:=\delta/N>0$, and denotes the approximate value of $u(n\tau)$ by $u_n^{(\tau)}$ for all $n\in\NN$. The Magnus-type integrator, which we will derive in detail in Section \ref{sec:magnus}, takes then the form
\begin{equation}\label{eq:magnus0}
\begin{aligned}
u_{n+1/2}^{(\tau)} &:= \left\{\begin{array}{ll}
\varphi((n+1/2)\tau-\delta) & \text{for}\:\ n=0,1,\dots,N-1, \medskip\\
\ee^{\frac\tau 2 Q\left(
\sum_{\ell=0}^{\lfloor\frac{\delta-\epsilon}{\tau}\rfloor}
\kappa_{\ell,\tau} F_{\ell,\tau}\left(u_{n-2N+\ell}^{(\tau)}\right)
\right)}u_{n-N}^{(\tau)} & \text{for}\:\ n\ge N,
\end{array}\right. \\
u_{n+1}^{(\tau)} &:= \ee^{\tau 
Q\left(
\sum_{\ell=0}^{\lfloor\frac{\delta-\epsilon}{\tau}\rfloor}
\kappa_{\ell,\tau} F_{\ell,\tau}\left(u_{n+\ell+1/2}^{(\tau)}\right)
\right)
}
u_n^{(\tau)}
\end{aligned}
\end{equation}
for $n\geq0$, where possible negative indices refer to the corresponding values of the history function, i.e., $u_n^{(\tau)}=\varphi(n\tau)$ for $n\le 0$, the half-indexed terms $u^{(\tau)}_{n+1/2}$ are auxiliary values (essentially approximating $u((n+1/2)\tau-\delta)$ for an application of the midpoint rule), and the expressions $\kappa_{\ell,\tau} F_{\ell,\tau}$ stem from the approximation of $F$. \medskip

Under appropriate smoothness and uniform exponential bounds on the generators $Q(\cdot)$ involved, we will prove the second-order convergence of the Magnus-type integrator \eqref{eq:magnus0} when applied to the abstract delay equation \eqref{eq:delay} on Banach spaces. Moreover, we show that the method inherits the invariance properties of the generators $Q(\cdot)$ (e.g.~positivity). \medskip

The paper is organised as follows. In Section \ref{sec:nonauto} we introduce the abstract setting of evolution equations like \eqref{eq:ncp}. In Section \ref{sec:magnus} we recall the original Magnus integrators, summarise the main results from the literature regarding its application to nonautonomous Cauchy problems \eqref{eq:ncp}, and then describe how we adapt this method to our case where we have no a priori knowledge of the operators $A(t)$ on the time interval $(\epsilon,\infty)$. \\
Section \ref{sec:conv} contains our main result, Theorem \ref{thm:main} on the Magnus-type integrator's second-order convergence when applied to the quasilinear delay equation \eqref{eq:delay}. Some of the assumptions needed for this convergence (or even the existence of the solution to the delay equation) -- especially a uniform exponential bound for the semigroups generated by the operators $Q(w)$ -- are typically not naturally achievable for \emph{all} $w\in X$. However, many problems admit invariants and have qualitative preservation features (e.g., positivity of the solutions), and we shall show that our Magnus-type integrator naturally exploits such invariants, allowing us to restrict our assumptions to some smaller, invariant closed subset $W$ of $X$ that we assume our initial history $\varphi$ runs within. This will allow us to prove in parallel that the Magnus-type integrator and the exact solution both exist for a positive amount of time (namely, $\epsilon$), with the former never leaving the invariant set. That in term will imply the second-order convergence on this time interval, implying that the solution itself stays in the invariant set as well. Iterating these arguments, we obtain our results for any compact time interval.\\
In Section \ref{sec:exm} we use a space-dependent epidemic model to illustrate the power of invariants (in this case both total population size and positivity) in ensuring second-order convergence.

\section{Autonomous and nonautonomous evolution equations}
\label{sec:nonauto}

In this section we introduce the notions necessary to understand the Magnus integrator and our approach to the error bounds. Throughout we shall assume that $X$ is a Banach space. Our main references are \citet{Engel-Nagel} and \citet{Nickel}.

\begin{definition}\label{def:semigroup}
A family $(\ee^{tA})_{t\ge 0}\subset \Ell(X)$ of bounded linear operators on $X$ is said to be a strongly continuous semigroup generated by the linear, closed, and densely defined operator $(A,D(A))$ if the following holds:
\begin{enumI}
\item $\ee^{0A}=\Id$, the identity operator in $X$,
\item $\ee^{(t+s)A}=\ee^{tA}\ee^{sA}$ for all $t,s\ge 0$,
\item the function $[0,\infty)\ni t\mapsto \ee^{tA}v$ is continuous for all $v\in X$,
\item there exists $\lim\limits_{t\to 0+}\tfrac 1t(\ee^{tA}v-v)=Av$ for all $v\in D(A)$.
\end{enumI}
\end{definition}
This strongly continuous semigroup describes the solution to the Abstract Cauchy Problem
\begin{equation}\tag*{$(\mr{ACP})_{x}$}
\left\{
\begin{aligned}
\dt \widehat u(t) &= A\widehat u(t), & t\in[0,\infty), \\
\widehat u(0) &= x\in X
\end{aligned}
\right.
\end{equation}
in the sense that $\widehat u(t)=\ee^{tA}x$. It is known \citep[see, e.g.,][Prop.~I.5.5]{Engel-Nagel} that such semigroups are exponentially bounded, i.e., there exist $M\ge 1$ and $\omega\in\RR$ such that $\|\ee^{tA}\|\leq M e^{\omega t}$ for all $t\geq 0$. We speak about a contraction semigroup if $M=1$ and $\omega=0$ can be chosen. A linear operator $(A,D(A))$ is called dissipative if $\|(\lambda-A)x\|\ge\lambda\|x\|$ holds for all $\lambda>0$ and $x\in D(A)$. \medskip

In many processes, however, one cannot assume the generator $(A,D(A))$ in $\AC$ to be constant in time, leading to so-called nonautonomous Cauchy problems, or $\NC$ for short.
Let $(A(t),D)$ be a linear operator on $X$ for every $t\in\RR$. Furthermore, let $x\in X$ and $s\in\RR$ be given. Then we consider the following nonautonomous problem for the differentiable unknown function $\widehat u\colon \RR\to X$:
\begin{equation}\tag*{$(\mr{NCP})_{s,x}$}
\left\{
\begin{aligned}
\dt \widehat u(t) &= A(t)\widehat u(t), & t\geq s, \\
\widehat u(s) &= x.
\end{aligned}
\right.
\end{equation}

The following definitions are based on \citet[Section VI.9]{Engel-Nagel}.

\begin{definition}
A continuous function $\widehat u\colon [s,\infty)\to X$ is called a (classical) solution of $\NC_{s,x}$ if $\widehat u\in\C^1(\RR,X)$, $\widehat u(t)\in D$ for all $t\ge s$, $\widehat u(s)=x$, and $\dt \widehat u(t)=A(t)\widehat u(t)$ for all $t\geq s$.
\end{definition}

\begin{definition}
For a family $(A(t),D)_{t\in\RR}$ of linear operators on the Banach space $X$, the nonautonomous Cauchy problem $\NC$ is called well-posed with regularity subspaces $(Y_s)_{s\in\RR}$ if the following holds.
\begin{enumI}
\item \emph{Existence}: For all $s\in\RR$ the subspace
\begin{equation*}
Y_s:=\big\{y\in X\colon\text{there exists a solution } \widehat u_s(\cdot,y) \text{ for } \NC_{s,y}\big\}\subset D
\end{equation*}
is dense in $X$.
\item \emph{Uniqueness}: For every $y\in Y_s$ the solution $\widehat u_s(\cdot,y)$ is unique.
\item \emph{Continuous dependence}: The solution depends continuously on $s$ and $x$, i. e., if $s_n\to s\in\RR$, $\widehat u_n\to y\in Y_s$ with $\widehat u_n\in Y_{s_n}$ then we have $\|\overline u_{s_n}(t,u_n)-\overline u_s(t,y)\|\to 0$ uniformly for $t$ in compact subsets of $\RR$,
where
\begin{equation*}
\overline u_r(t;y)=\left\{\begin{array}{ll}
\widehat u_t(t,y)&\:\:\text{if}\:\: r\le t, \\
y&\:\:\text{if}\:\: r>t.
\end{array}\right.
\end{equation*}
If, in addition, there exist constants $M\ge 1$ and $\omega\in\RR$ such that
\begin{equation*}
\|\widehat u_s(t,y)\|\le M\ee^{\omega(t-s)}\|y\|
\end{equation*}
for all $y\in Y_s$ and $t\ge s$, then $\NC$ is called well-posed with exponentially bounded solutions.
\end{enumI}
\end{definition}

\begin{definition}\label{def:evolfam}
A family $(\U(t,s))_{t\ge s}$ of linear, bounded operators on a Banach space $X$ is called an (exponentially bounded) evolution family if
\begin{enumI}
\item $\U(t,r)\U(r,s)=\U(t,s) \text{ and } \U(t,t) = \Id$ for all $t\ge r\ge s\in\RR$,
\item the map $(t,s)\mapsto \U(t,s)$ is strongly continuous,
\item $\|\U(t,s)\|\le M\ee^{\omega(t-s)}$ for some $M\ge 1$, $\omega\in\RR$ and all $t\ge s\in\RR$.
\end{enumI}
An evolution family is said to be \emph{contractive} if we can choose $\omega=0$ and $M=1$, and \emph{quasi-contractive} if we can choose $M=1$ for some $\omega\in\RR$.
\end{definition}

\begin{definition}\label{def:solving}
An evolution family $(\U(t,s))_{t\ge s}$ is called an evolution family solving $\NC$, if for every $s\in\RR$ the regularity subspace
\begin{equation*}
Y_s:=\big\{y\in X\colon [s,\infty)\ni t\mapsto \U(t,s)y \text{ solves }\NC_{s,y}\big\}
\end{equation*}
is dense in $X$.
\end{definition}

We have the following result connecting the well-posedness of $\NC$ to the existence of a unique evolution family solving it.

\begin{theorem}[{\citet[Prop.~2.5]{Nickel}}]\label{thm:well-posed}
Let $X$ be a Banach space, $(A(t),D(A(t)))_{t\in\RR}$ a family of linear operators on $X$. Then the nonautonomous Cauchy problem $\NC$ is well-posed if and only if there exists a unique evolution family $(\U(t,s))_{t\ge s}$ solving $\NC$.
\end{theorem}

As mentioned in the Introduction, our goal is to rephrase \eqref{eq:delay} as a nonautonomous Cauchy problem on the time interval $[0,\infty)$. Of great help in establishing well-posedness is the following consequence of a result by \citet[Thm. 4]{Kato53}.

\begin{theorem}
Let $J\subset\RR$ be a closed interval, and $(A(t),D)_{t\in J}$ a family of generators of contraction semigroups on the Banach space $X$ with common domain satisfying $A(\cdot)x\in~C^1(J,X)$ for all $x\in D$. Then the nonautonomous Cauchy problem $\dt u(t)=A(t)u(t)$ is well-posed with regularity subspaces $Y_s=D$ ($s\in J$) and admits a contractive evolution family $(\U(t,s))_{t\geq s; t,s\in J}$. In particular, for any $x\in D$ and $s\in J$ the initial condition $u(s)=x$ leads to a unique (classical) solution $(u(t))_{s\leq t\in J}$ with $u(t)\in D$ for all $s\leq t\in J$. 
\end{theorem}
\begin{remark}
The contractivity of the evolution family is actually part of the earlier Theorem 2 in the same paper.\\
Also, note that Kato's result talks of evolution families and well-posedness on a possibly bounded closed interval $J$ instead of $\RR$, and it is not immediately clear how the two can be connected. Let $J=[a,b]$ with $a,b\in\RR$. Since each $A(t)$ ($t\in J$) is a generator, it makes sense to extend $A(\cdot)$ to $\RR$ by setting $A(s)=A(a)$ when $s\leq a$ and $A(s)=A(b)$ when $b\leq s$. Then the evolution family corresponding to the extended interval can obviously be restricted to $J$ with all the properties preserved. But also the evolution family $(\U(t,s))_{s\leq t; t,s\in J}$ corresponding to the restricted problem has a natural extension to $\RR$ as follows.\\ Denote by $(T_a(\tau))_{\tau\geq0}$ the contraction semigroup generated by $A(a)$, and by $(T_b(\tau))_{\tau\geq0}$ the contraction semigroup generated by $A(b)$.
If $s\leq a$, then let $\mc{U}(t,s):=\mc{U}(t,a)T_a(a-s)$ and if $b\leq t$, then let $\mc{U}(t,s):=T_b(t-b)\mc{U}(b,s)$. This defines the evolution family for all $s\leq t$ in a way compatible with Definition \ref{def:evolfam} ($\omega$ may have to be replaced with $\max\{\omega,0\}$).
\end{remark}

An easy rescaling argument yields that the same holds if instead of generators of contraction semigroups we consider a family $(A(t))_{t\in J}$ such that the generated semigroups are uniformly quasi-contractive.

\begin{corollary}\label{cor:wellposed}
Let $J\subset\RR$ be a closed interval, and $(A(t),D)_{t\in J}$ a family of generators of uniformly quasi-contractive semigroups on the Banach space $X$ with common domain satisfying $A(\cdot)x\in~C^1(J,X)$ for all $x\in D$. Then the nonautonomous Cauchy problem $\dt u(t)=A(t)u(t)$ is well-posed with regularity subspaces $Y_s=D$ ($s\in J$) and admits a quasi-contractive evolution family $(\U(t,s))_{s\leq t; t,s\in J}$. In particular, for any $x\in D$ and $s\in J$ the initial condition $u(s)=x$ leads to a unique (classical) solution $(u(t))_{s\leq t\in J}$ with $u(t)\in D$ for all $s\leq t\in J$. 
\end{corollary}
Hence, the unique solution to $\NC_{s,x}$ has the form
\begin{equation}\label{eq:solution}
\widehat u(t)=\U(t,s)x
\end{equation}
for all $t\ge s$. Our aim is to approximate the solution $u$ to problem \eqref{eq:delay}, rewritten as the nonautonomous problem \eqref{eq:ncp}, at certain time levels. To do this we introduce the Magnus-type integrator in the next section.

\section{Magnus-type integrator}
\label{sec:magnus}

We saw in the Introduction that the delay equation \eqref{eq:delay} could formally be written as the nonautonomous abstract Cauchy problem $\NC_{0,\varphi(0)}$ with the solution-dependent operator $A(t)=Q(F(\t u))$. Whilst well-posed autonomous abstract Cauchy problems have their solutions given through of a one-parameter strongly continuous semigroup, problem $\NC_{0,\varphi(0)}$ -- if well-posed and $\varphi(0)$ is in the regularity subspace $Y_0$ -- has its solution given through a unique two-parameter evolution family $(\U(t,s))_{t\ge s}$ (cf.~Theorem \ref{thm:well-posed}). Since the exact form of the evolution family $\U$ is usually unknown, we need to approximate the solution in \eqref{eq:solution}. To this end we define an arbitrary $N\in\NN$ and take the time step $\tau:=\delta/N>0$. Then the approximate value of $\widehat u(t_n)$ at time levels $t_n:=n\tau$, $n\in\NN$, is denoted by $\widehat u_n^{(\tau)}$. \medskip

In case of a finite dimensional space $X$, the evolution family $\U$ is the exponential of a bounded operator $\Omega(t,s)$ for $t\ge s\ge 0$. Magnus showed in \citet{Magnus} that the bounded operator $\Omega(t,s)$ could be expressed by the integral of an infinite series, called Magnus series. Casas and Iserles showed in \citet{Casas-Iserles} that the appropriate truncations of the series led to convergent approximations. The further approximation of the integral terms yields the Magnus-type integrators. More precisely, by the property of the evolution family we have
\begin{equation*}
\widehat u((n+1)\tau)=\U((n+1)\tau,s)x=\U((n+1)\tau,n\tau)\U(n\tau,s)x = \U((n+1)\tau,n\tau)\widehat u(n\tau)
\end{equation*}
for all $n\in\NN$. By approximating $\U((n+1)\tau,n\tau)$ by $\exp\left(\int_0^\tau A(n\tau+\zeta)\dd\zeta\right)$ as in \citet{Casas-Iserles}, and then by $\ee^{\tau A((n+1/2)\tau)}$ with the midpoint quadrature rule, we arrive at the formula of the simplest Magnus integrator
\begin{equation}\label{eq:magnusA}
\widehat u_{n+1}^{(\tau)} = \ee^{\tau A((n+1/2)\tau)}\widehat u_n^{(\tau)}
\end{equation}
for all $n\in\NN$ with $\widehat u_0^{(\tau)}=x$.
\medskip

In case of an infinite dimensional Banach space $X$, we consider formally the same formula \eqref{eq:magnusA} where the exponential refers to the strongly continuous semigroup generated by the corresponding operator (cf.~Definition \ref{def:semigroup}). \medskip

Since the Magnus integrator \eqref{eq:magnusA} gives only an approximation to the exact solution at time $t_n=n\tau$ for all $n\in\NN$, it is necessary to show that the approximate value converges to the exact value as the time step $\tau=\delta/N$ tends to zero, or equivalently, $N\in\NN$ tends to infinity. As is usual, we will want to show that our numerical scheme yields a good approximation on any compact time interval $[0,T]$.

\begin{definition}\label{def:conv}
Let $u$ denote the exact solution to problem \eqref{eq:delay}. Its approximation $u_n^{(\tau)}$ (or, equivalently, the corresponding numerical method) is called convergent of order $p>0$, if there exists $C>0$ such that $\|u(n\tau)-u_n^{(\tau)}\|\le C\tau^p$ holds for all $n\in\NN$ and $\tau=\delta/N>0$ with $n\tau\in[0,T]$, where the constant $C$ is independent of $n$ and $\tau$ but may depend on $n\tau$.
\end{definition}
We note that we defined the convergence for the sequence $\tau=\delta/N$ of the time steps in order to simplify our proofs, which could be done for all sequences $\tau_k>0$ and $n_k\in\NN$ with $n_k\tau_k\in[0,T]$ by introducing a more complicated formalism. Since one has an initial history function (or a set of data) on the time interval $[-\delta,0]$, it is natural to choose a time step that is compatible with it. \medskip

In what follows we present two results from the literature about the Magnus integrator \eqref{eq:magnusA} for nonautonomous problems, since we will use them in our analysis.
\begin{theorem}[{\citet[Thm.~2]{Gonzalez-etal}}]\label{thm:ostermann}
Let $(X,\|\cdot\|_X)$ and $(D,\|\cdot\|_D)$ be Banach spaces with $D$ densely embedded in $X$. We suppose that the closed linear operator $A(t)\colon D\to X$ is uniformly sectorial for $t\in[0,T]$. Moreover, we assume that the graph norm of $A(t)$ and the norm in $D$ are equivalent. We also assume that $A\in\C^1([0,T],\Ell(D,X))$, and in particular there then exists a constant $L_A>0$ such that 
\begin{equation*}
\|A(t)-A(s)\|_{\Ell(D,X)}\le L_A(t-s)
\end{equation*}
holds for all $0\le t\le s\le T$. Moreover, we introduce the notations
\begin{equation}\label{eq:g}
\begin{aligned}
g_n(t) &= \big(A(t)-A((n+1/2)\tau)\big)u(t), \quad t\in[n\tau,(n+1)\tau], \\
\|g_n\|_{X,\infty} &= \max\{\|g_n(t)\|_X\colon t\in[n\tau,(n+1)\tau]\}, \\
\|g\|_{X,\infty} &= \max\{\|g_n\|_{X,\infty}\colon n\in\NN, (n+1)\tau\in[0,T]\},
\end{aligned}
\end{equation}
and corresponding notations will also be used with $D$ instead of $X$.
Then the Magnus integrator \eqref{eq:magnusA} applied to problem $\NC_{0,\varphi(0)}$ is convergent of second order, that is, there exists a constant $C>0$, being independent of $n$ and $\tau$, such that the following estimate holds for the global error:
\begin{equation}\label{eq:rhs}
\|\widehat u(n\tau)-\widehat u_n^{(\tau)}\|\le C\tau^2(\|g'\|_{D,\infty}+\|g''\|_{X,\infty})
\end{equation}
for all $n\tau\in[0,T]$, provided that the quantities on the right-hand side are well-defined.
\end{theorem}

The following theorem is essentially the quasi-contractive, continuously differentiable special case of the consistency result from the proof of \citet[Thm.~3.2]{Batkai-Sikolya}, more specifically inequality (5) therein. It hinges on the fact that Corollary \ref{cor:wellposed} implies that the well-posedness, stability and local Hölder continuity conditions of that theorem are automatically satisfied.

\begin{theorem}[{\citet[Eq.~(5) in proof of Thm.~3.2]{Batkai-Sikolya}}]\label{thm:batkai}
We consider the problem $\NC$ on the Banach space $X$ and suppose the following.
\begin{enumA}
\item There exists a $c\in\RR$ such that $\|\ee^{h A(r)}\|\le \ee^{ch}$ for all $r\in\RR$ and $h >0$. Further, $A(t)=A+V(t)$, where $A$ is the generator of a strongly continuous semigroup, and $V(t)\in\Ell(X)$ for all $t\in\RR$.
\item The map $t\mapsto V(t)$ is continuously differentiable as a map from $[a,b]$ to $\Ell(X)$ for some $a,b\in\RR$.
\end{enumA}
Then for all $s,h\in\RR$ with $a\leq s<s+h\leq b$ we have the following estimate:
\begin{equation}\label{eq:magnuscons}
\|\U(s+h,s)-\ee^{\tau A(s+h/2)}\|\le L_{a,b}\ee^{c h}h^{2},
\end{equation}
where $L_{a,b}$ is the Lipschitz constant of $t\mapsto V(t)$ on $[a,b]$.
\end{theorem}

To illustrate the process of obtaining our Magnus-type integrator, we shall first present a special case that already exhibits some of the new ideas involved, and then we indicate what further changes are needed to accommodate for the general case.

\begin{example}\label{exm:point-delay}
In this example we shall focus on the special case $F(\t u):=u(t-\delta)$, i.e., when the delay is concentrated on a specific point of the history function. Our equation \eqref{eq:delay} then takes the form
\begin{equation}\label{eq:delay_0}\tag{QDE'$_\varphi$}
\left\{
\begin{aligned}
\dt u(t) &= Q(u(t-\delta))u(t), & t\in[0,\infty), \\
u(s) &= \varphi(s), & s\in[-\delta,0].
\end{aligned}
\right.
\end{equation}
In this particular case, we have $A(t)=Q(u(t-\delta))$ in \eqref{eq:ncp}. Recall that $u(t-\delta)=\varphi(t-\delta)$ for $t\in[0,\delta]$, however, its value is unknown for $t>\delta$. Thus, we need to approximate $A((n+1/2)\tau)=Q(u((n+1/2)\tau-\delta))$ in the Magnus integrator \eqref{eq:magnusA} to obtain a working method. Now the natural idea would be to use the appropriate $\widehat u^{(\tau)}_*$, however $(n+1/2)\tau-\delta$ falls exactly between two points of our time grid, and has to be obtained via further approximation. We therefore introduce the corresponding term as an auxiliary value, obtained via another Magnus step with half time step. In full detail, we have the following.
\begin{definition*}
Let $N\in\NN$ be an arbitrary integer and $\tau:=\delta/N$. Then the \textbf{Magnus-type integrator for point-delay}, which yields an approximation $u_n^{(\tau)}$ to the solution $u(n\tau)$ of the point-delay equation \eqref{eq:delay_0} at time levels $n\tau$ is given as follows. We introduce the notation $u_n^{(\tau)}:=\varphi(n\tau)$ for $n=-N,\dots,0$. Then for $n\ge 0$, we have the recursion
\begin{equation}\label{eq:magnus1}
\begin{aligned}
u_{n+1/2}^{(\tau)} &:= \left\{\begin{array}{ll}
\varphi((n+1/2)\tau-\delta) & \text{for}\:\ n=0,1,\dots,N-1, \medskip\\
\ee^{\frac\tau 2 Q(u_{n-2N}^{(\tau)})}u_{n-N}^{(\tau)} & \text{for}\:\ n\ge N,
\end{array}\right. \\
u_{n+1}^{(\tau)} &:= \ee^{\tau Q(u_{n+1/2}^{(\tau)})}u_n^{(\tau)}.
\end{aligned}
\end{equation}
\end{definition*}

Note that since they are essentially auxiliary values, the way we indexed the terms $u_{n+1/2}^{(\tau)}$ is not indicative of the time layer they correspond to (which would actually be $n+1/2-N$). Rather, the indices reflect the natural order in which one would execute the algorithm, i.e.,
\[
\ldots, u_{n}^{(\tau)}\enspace;\enspace u_{n+1/2}^{(\tau)},u_{n+1}^{(\tau)}\enspace;\enspace u_{n+3/2}^{(\tau)},u_{n+2}^{(\tau)}\enspace;\enspace u_{n+5/2}^{(\tau)}\ldots,
\]
despite being able to compute $u_{n+1/2}^{(\tau)}$ already at earlier stages.
\end{example}

\medskip

In the general case of the quasilinear delay evolution equation \eqref{eq:delay}, we have $A(t)=Q(F(\t u))$ in \eqref{eq:ncp}. If we want to apply the Magnus integrator \eqref{eq:magnusA}, we need to be able to approximate $F(\t u)$. Even for $\t u$, all we have available is an approximation to some of its values, corresponding to the time levels in our grid. Thus we first need a discretisation of $F$ itself using only discrete values of $\t u$, and then use the approximation of those values in the final form of our method. The simplest approach is to make the discretisation of $F$ compatible with the original time grid, and use the same auxiliary Magnus step as in the previous Example \ref{exm:point-delay}. \medskip

To this end we introduce the approximation of $F$ in the following form
\begin{equation}\label{eq:Fapprox}
F(\xi)\approx\sum\limits_{\ell=0}^{\lfloor\frac{\delta-\epsilon}{\tau}\rfloor}\kappa_{\ell,\tau}F_{\ell,\tau}(\xi(-\delta+\ell\tau))
\end{equation}
for appropriate elements $\xi\in\C([-\delta,0],X)$, weights $\kappa_{\ell,\tau}\in\RR$, and functions $F_{\ell,\tau}\colon X\to X$ having properties to be detailed in Section \ref{sec:conv}. We rewrite now \eqref{eq:delay} as a nonautonomous problem \eqref{eq:ncp} with $A(t)=Q(F(\t u))$, apply the Magnus integrator \eqref{eq:magnusA}, and approximate $F$ as in \eqref{eq:Fapprox} to obtain for all $n\in\NN$:
\begin{equation}\label{eq:magnus_long1}
\begin{aligned}
& u((n+1)\tau)
\approx \ee^{\tau A((n+1/2)\tau)}u(n\tau) = \ee^{\tau Q\left(F\left(\prescript{}{(n+1/2)\tau}u\right)\right)}u(n\tau) \\
&\approx \ee^{\tau Q\left(\sum_{\ell=0}^{\lfloor\frac{\delta-\epsilon}{\tau}\rfloor}\kappa_{\ell,\tau}F_{\ell,\tau}\left(\prescript{}{(n+1/2)\tau}u(-\delta+\ell\tau)\right)\right)}u(n\tau)
= \ee^{\tau Q\left(\sum_{\ell=0}^{\lfloor\frac{\delta-\epsilon}{\tau}\rfloor}\kappa_{\ell,\tau}F_{\ell,\tau}\left(u((n+1/2)\tau-\delta+\ell\tau)\right)\right)}u(n\tau),
\end{aligned}
\end{equation}
where we used the definition of the history function in the last step. We approximate next the intermediate values using another Magnus step with time step $\tau/2$ but now by taking the left-rectangle rule when approximating the integral $\int_0^{\tau/2} A(n'\tau+\zeta)\dd\zeta\approx\frac\tau 2A(n'\tau)$ in the exponent, cf.~\citet{Casas-Iserles}:
\begin{equation}\label{eq:magnus_long2}
\begin{aligned}
&u((n+\ell)\tau-\delta+\tau/2)
\approx \ee^{\frac\tau 2A((n+\ell)\tau-\delta)}u((n+\ell)\tau-\delta)=\ee^{\frac\tau 2Q\left(F\left(\prescript{}{(n+\ell)\tau-\delta}u\right)\right)}u((n+\ell)\tau-\delta) \\
&\approx \ee^{\frac\tau 2Q\left(\sum_{k=0}^{\lfloor\frac{\delta-\epsilon}{\tau}\rfloor}\kappa_{k,\tau}F_{k,\tau}\left(u((n+\ell)\tau-\delta-\delta+k\tau)\right)\right)}u((n+\ell)\tau-\delta)
\end{aligned}
\end{equation}
for all $\ell=0,\dots,\lfloor\frac{\delta-\epsilon}{\tau}\rfloor$ with $n+\ell\geq N$. Observe that formula \eqref{eq:magnus_long2} can be reindexed by using $n\in\NN$ instead of $n+\ell$. However, since $A(\cdot)$ is not defined for negative times, whenever $n+\ell< N$, the values at the corresponding intermediate time levels should be obtained from the initial history function $\varphi$ instead. By combining \eqref{eq:magnus_long1} and \eqref{eq:magnus_long2}, we thus obtain the following definition.

\begin{definition}
Let $N\in\NN$ be an arbitrary integer and $\tau:=\delta/N$. Then the \textbf{Magnus-type integrator}, which yields an approximation $u_n^{(\tau)}$ to the solution $u(n\tau)$ of the delay equation \eqref{eq:delay} at time levels $n\tau$ is given as follows. As before, we use the notation $u_n^{(\tau)}:=\varphi(n\tau)$ for $n=-N,\dots,0$. Then for $n\ge 0$, we have the recursion
\begin{equation}\label{eq:magnus}
\begin{aligned}
u_{n+1/2}^{(\tau)} &:= \left\{\begin{array}{ll}
\varphi((n+1/2)\tau-\delta) & \text{for}\:\ n=0,1,\dots,N-1, \medskip\\
\ee^{\frac\tau 2 Q\left(
\sum_{\ell=0}^{\lfloor\frac{\delta-\epsilon}{\tau}\rfloor}
\kappa_{\ell,\tau} F_{\ell,\tau}\left(u_{n-2N+\ell}^{(\tau)}\right)
\right)}u_{n-N}^{(\tau)} & \text{for}\:\ n\ge N,
\end{array}\right. \\
u_{n+1}^{(\tau)} &:= \ee^{\tau 
Q\left(
\sum_{\ell=0}^{\lfloor\frac{\delta-\epsilon}{\tau}\rfloor}
\kappa_{\ell,\tau} F_{\ell,\tau}\left(u_{n+\ell+1/2}^{(\tau)}\right)
\right)
}
u_n^{(\tau)},
\end{aligned}
\end{equation}
where the exact properties of the weights $\kappa_{\ell,\tau}$ and of the functions $F_{\ell,\tau}$ will be presented in Section \ref{sec:conv}.
\end{definition}

For finite dimensional spaces $X$, it was shown in \citet{Csomos} that the Magnus-type integrator \eqref{eq:magnus1} was convergent of second order. Our present aim is to generalise this result to any Banach space $X$ and to general $F$ for the Magnus-type integrator \eqref{eq:magnus}.

\begin{remark}\label{rem:point-delay}
To fit Example \ref{exm:point-delay} in this general formulation, set $\kappa_{0,\tau}=1$ and $\kappa_{\ell,\tau}=0$ for $\ell\geq1$, with $F_{\ell,\tau}$ the identity for all $\ell\geq0$ .
\end{remark}

\begin{example}\label{exm:integral}
Let us consider a delay term where a fixed delay time interval uniformly governs the dynamics. More specifically, we shall take a closer look at the delay when
\begin{equation*}
F(\xi)=\frac{2}{\delta}
\int_{-\delta}^{-\delta/2} \xi(s)\dd s,
\end{equation*}
i.e, we have
\begin{equation*}
Q\left(\frac{2}{\delta}
\int_{-\delta}^{-\delta/2} u(t+s)\dd s
\right) \quad\text{with}\quad \delta>0
\end{equation*}
in \eqref{eq:delay}.

The first thing to note is that whenever $\delta/\tau=N$ is divisible by $2$, the endpoint $-\delta/2$ of the integral also falls on the discretisation grid, significantly simplifying things, and the odd $N$'s have to be treated slightly differently to fit the general framework. Alternatively, we could simply adapt the general scheme to this special situation by only considering even values for $N$, but we shall detail the odd $N$ case nevertheless.\\
In contrast to the point-interaction delay in Example \ref{exm:point-delay} where we could directly substitute the computed values $u^{(\tau)}_*$ into the delay term, we here have an integral $\int_{-\delta}^{-\delta/2} u(t+s)\dd s$ that itself has to be numerically approximated using some appropriate quadrature. Taking into consideration the order of the error that we need to achieve for the recursive inequalities to result in a second-order Magnus-type integrator, the error of quadrature has to be of the magnitude of $\tau^2$.

So for even $N$, we use the composite trapezoidal rule with nodes $-\delta+\ell\tau$ for $\ell=0,\dots,N/2$ as
\begin{equation*}
\frac{2}{\delta}\int_{-\delta}^{-\delta/2} u(t+s)\dd s\approx \frac 1N\left(u(t-\delta)+2\sum\limits_{\ell=1}^{N/2-1}u(t-\delta+\ell\tau)+u(t-\delta/2)\right).
\end{equation*}
The Magnus-type integrator thus has the form for $n\ge 0$:
\begin{equation*}
\begin{aligned}
u_{n+1/2}^{(\tau)} &:= \left\{\begin{array}{ll}
\varphi((n+1/2)\tau-\delta), &  n=0,\dots,N-1, \medskip\\
\ee^{\frac\tau 2 Q\left(
\frac{1}{N}
\left(
u_{n-2N}^{(\tau)}+u_{n-2N+N/2}^{(\tau)}+2\sum\limits_{\ell=1}^{N/2-1}u_{n-2N+\ell}^{(\tau)}
\right)
\right)}u_{n-N}^{(\tau)}, & n\ge N,
\end{array}\right. \\
u_{n+1}^{(\tau)} &:= \ee^{\tau 
Q\left(
\frac{1}{N}\left(
u_{n+1/2}^{(\tau)}+u_{n+N/2+1/2}^{(\tau)}+2\sum\limits_{\ell=1}^{N/2-1}u_{n+\ell+1/2}^{(\tau)}
\right)
\right)
}
u_n^{(\tau)}.
\end{aligned}
\end{equation*}
It has the form \eqref{eq:magnus} with weights $\kappa_{0,\tau}=\kappa_{N/2,\tau}=1/N$ and $\kappa_{\ell,\tau}=2/N$ for $\ell=1,\dots,N/2-1$ where each $F_{\ell,\tau}$ is the identity. We remark that the weights $\kappa_{\ell,\tau}$ sum up to $1$ (cf.~\eqref{eq:weights} later on).\\
For odd values of $N$, the point $\delta/2$ is not part of the grid, so we have to use the truncated approximation
\begin{align*}
\frac{2}{\delta}\int_{-\delta}^{-\delta/2-\tau/2} u(t+s)\dd s\approx \frac 1N\left(u(t-\delta)+2\sum\limits_{\ell=1}^{(N-1)/2}u(t-\delta+\ell\tau)\right)
\end{align*}
instead, again with error of order $\tau^2$, i.e., $\kappa_{0,\tau}=1/N$ and $\kappa_{\ell,\tau}=2/N$ for $\ell=1,\dots,(N-1)/2$.
 \medskip

It is reasonable to wonder what happens if the delays above are not given in the convenient form where we are looking at a convex combination/normalised integral, for instance, if we were dealing with a delay of the form $Q\left(\int_{-\delta}^{-\delta/2} u(t+s)\dd s\right)$. Actually, this is not really an issue, as we may then define $\mc{Q}(x):=Q(x\delta/2)$ and thereby revert to the convex combination case detailed above. It is rather a matter of convenience, allowing our Assumptions detailed below to be formulated in a less cumbersome way.
\end{example}

\section{Convergence}
\label{sec:conv}

This section contains our main result regarding the second-order convergence of Magnus-type integrator \eqref{eq:magnus} applied to the quasilinear delay equation \eqref{eq:delay}. We have already seen how problem \eqref{eq:delay} can formally be written as the nonautonomous problem $\NC_{0,\varphi(0)}$. Hence, Definition \ref{def:conv} of the convergent approximation remains valid also for the Magnus-type integrator \eqref{eq:magnus} applied to the delay problem \eqref{eq:delay}.
From now on we use the notations $\dot h_-(0)$ and $\dot h_+(0)$ for the left and right derivatives of a function $h$ at zero, respectively. \\
We will need the following list of assumptions, with their motivation following right after.

\begin{mainass}
Let $X$ be a Banach space, $W\subset X$ a closed subset and $D\subset X$ a dense subspace.
\begin{enumI}
\item\label{Q} We have $Q(x)=Q_0+\Q(x)=Q_0+\widehat{Q}(x)+c\Id$ for all $x\in X$, where $c\geq0$, $(Q_0,D)$ generates a contraction semigroup on $X$, $\left(x\mapsto \widehat{Q}(x)\right)\in\C(X,\Ell(X))$, and the operators $\widehat{Q}(w)\in\Ell(X)$ ($w\in W$) are all dissipative. We shall also assume $0\in\varrho(Q_0)$, which is no real added restriction as we may simply replace $Q_0$ by $Q_0-\veps\Id$ and $c$ by $c+\veps$ for some $\veps>0$.
\item[\hyt{$\mr{(\ref{smooth2})}_0$}{smooth}]
The function $\Q:X\to\Ell(X)$ is continuously differentiable on the set $W$.
\item
\label{smooth2} The function $\Q:X\to\Ell(X)$ is twice continuously differentiable on the set $W$.
\item\label{dissip} For any $w\in W\cap D$ the operator $\widehat{Q}(w)$ leaves $D$ invariant, and is bounded and dissipative on $(D,\|\cdot\|_D)$, where $\|x\|_D:=\|Q_0x\|_X+\|x\|_X$. In addition, $\Q:D\to\Ell(D)$ is continuously differentiable on $W\cap D$.
\item\label{analytic} The operator $(Q_0,D)$ is a sectorial operator generating an (analytic) contraction semigroup with sector $\Sigma_\alpha$ for some $\alpha>0$, and the operators $\widehat{Q}(w)\in\Ell(X)$ ($w\in W$) are all such that for any $|\phi|<\alpha$ the operator $\ee^{\ii\phi}\widehat{Q}(w)$ is dissipative.
\item[\hyt{$\mr{(\ref{ass:1})}_0$}{ass:0}] $F\in \C^1\left(\C([-\delta,0],X),X\right)$ with $F(u_1)=F(u_2)$ whenever $(u_1-u_2)|_{[-\delta,-\epsilon]}=0$.
\item\label{ass:1} $F\in \C^2\left(\C([-\delta,0],X),X\right)\cap \C^1\left(\C([-\delta,0],D),D\right)$ with $F(u_1)=F(u_2)$ whenever $(u_1-u_2)|_{[-\delta,-\epsilon]}=0$.
\item\label{ass:2} The functions $F_{\ell,\tau}:X\to X$ and weights $\kappa_{\ell,\tau}\in \mb{R}$ ($0\leq\ell\leq (\delta-\epsilon)/\tau$) are such that there exists a constant $L>0$ independent of $\tau$ such that 
\begin{equation}\label{eq:weights}
\sum_{\ell=0}^{\lfloor\frac{\delta-\epsilon}{\tau}\rfloor}
\left|\kappa_{\ell,\tau}\right|\leq L
\end{equation}
and the function
\begin{align}\label{eq:kvadra}
F_\tau(\xi):=\sum_{\ell=0}^{\lfloor\frac{\delta-\epsilon}{\tau}\rfloor}
\kappa_{\ell,\tau} F_{\ell,\tau}\left(\xi(-\delta+\ell\tau)\right)
\end{align}
satisfies
\begin{equation}\label{eq:kvadr_err}
\left\|
F(\xi)-F_\tau(\xi)
\right\|_X\leq \ms{C} \|\xi\|_{\C^2([-\delta,0],X)}\tau^2
\end{equation}
for some constant $\ms{C}\geq 0$ independent of $\tau$. In addition, there exists a constant $L_{F}>0$ such that $F$ and $F_{\ell,\tau}$ are Lipschitz with constant $L_F$ for all $\tau$ and $\ell$.
\item\label{ass:3} The set $W$ satisfies
\begin{equation*}
\left[ v\in(\C([-\delta,0],X)) \text{ and } v([-\delta,-\epsilon])\subset W\right]\Rightarrow F(v)\in W
\end{equation*}
and
\begin{equation}\label{eq:Winv}
\sum_{\ell=0}^{\lfloor\frac{\delta-\epsilon}{\tau}\rfloor}
\kappa_{\ell,\tau} F_{\ell,\tau}(W)\subset W
\end{equation}
 for all $\tau$.
\item\label{ass:4} The initial history function $\varphi$ is in $\C^2([-\delta,0],X)\cap \C^1([-\delta,0],D)$ and satisfies $\varphi(t)\in W$ for all $t\in[-\delta,0]$ and the boundary conditions
 \begin{align}
\label{eq:C1}
 \dot\varphi_-(0)&=Q(F(\varphi))\varphi(0),
 \\
 \label{eq:C2}
 \ddot\varphi_-(0)&=\left(\widetilde{Q}'(F(\varphi))F'(\varphi)\dot\varphi\right)\varphi(0)+Q(F(\varphi))^2\varphi(0).
  \end{align}
\end{enumI}
\end{mainass}

For non-autonomous equations, well-posedness is very much not easily guaranteed, and we chose to rely on the Corollary \ref{cor:wellposed} of Kato's result detailed earlier. This motivates the common domain prescribed by Assumption \eqref{Q}, the special form of the operators as bounded perturbations of the same (unbounded) $Q_0$, and the contractivity of the corresponding unperturbed semigroup. In addition, we need the smoothness of various terms of the equation provided by Assumptions \hyl{smooth} and \hyl{ass:0}.\\
The stronger versions \eqref{smooth2} and (part of) \eqref{ass:1} are the usual smoothness requirements on the various terms of the equation matching the desired order of the method.\\
In order to guarantee the preservation of second-order smoothness of the solution, we need to be able to work on the common domain $D$ as well, leading to Assumption \eqref{dissip} and the part of Assumption \eqref{ass:1} pertaining to $D$.\\
Assumption \eqref{analytic} mirrors the one in \citet{Gonzalez-etal} needed there to guarantee the second-order convergence of the Magnus method for autonomous equations.\\
Assumption \eqref{ass:2} is needed so the discretisation of $F$ does not ruin the order of the proposed method, Assumption \eqref{ass:4} on the initial history function guarantees the smooth transition of the solution at $t=0$, and the invariance conditions of Assumption \eqref{ass:3} make it possible to iterate the arguments beyond the initial $[0,\epsilon]$ solution window.

\begin{remark}
The condition \eqref{eq:Winv} is automatically satisfied when $W$ is convex, $\sum_{\ell=0}^{\lfloor\frac{\delta-\epsilon}{\tau}\rfloor} \kappa_{\ell,\tau}=1$ and $\kappa_{\ell,\tau}\geq 0$ and $F_{\ell,\tau}(W)\subset W$
 for all $\tau$ and $\ell$.
\end{remark}

The following results will show that under appropriate smoothness assumptions on $Q(\cdot)$ and the initial history function $\varphi$, the solution itself will exhibit similar smoothness properties on $X$ and $D$. Also, we shall show that Theorem \ref{thm:ostermann} is applicable to our setting.

The next two results show that our Assumptions imply that the semigroups generated by the operators $Q(w)$ ($w\in W$) are uniformly quasi-contractive, and uniformly sectorial as well.

\begin{lemma}\label{le:contraction}
Under Assumption \eqref{Q} each operator $Q(w)-c\Id$ ($w\in W$) with domain $D$ is the generator of a contraction semigroup on $X$.
In particular, the semigroups generated by $Q(w)$ are uniformly quasi-contractive, that is, $\|\ee^{\tau Q(w)}\|_{\Ell(X)}\le\ee^{\tau c}$ for all $\tau >0$ and $w\in W$.
\end{lemma}
\begin{proof}
The operators $\widehat Q(w)$ are by assumption dissipative, and with $Q_0$-bound 0 (as they are bounded operators). Thus the claims follow directly from \citet[Thm.~III.2.7]{Engel-Nagel} and the usual rescaling argument. 
\end{proof}

\begin{lemma}\label{le:sector}
Under Assumptions \eqref{Q} and \eqref{analytic} each operator $Q(w)-c\Id$ ($w\in W$) with domain $D$ is the generator of an analytic contraction semigroup on $\Sigma_\alpha$. In particular, the uniform sectoriality required in Theorem \ref{thm:ostermann} is satisfied by the family $(Q(w))_{w\in W}$.
\end{lemma}
\begin{proof}
The assumptions imply that for any $|\phi|<\alpha$ the operator $\ee^{\ii\phi}Q_0$ generates a contraction semigroup and $\ee^{\ii\phi}\widehat{Q}(w)$ is dissipative with $Q_0$-bound 0, hence each $\ee^{\ii\phi}(Q(w)-c\Id)$ generates a contraction semigroup as well by \citet[Thm.~III.2.7]{Engel-Nagel}. To see that the semigroup $(S_w(\tau))_{\tau\in\Sigma_\alpha}$ generated by $Q(w)-c\Id$ is analytic on $\Sigma_\alpha$, we use the semigroup property and the fact that \citet[Thm.~III.2.14]{Engel-Nagel} implies that $(Q_0-\veps \Id)+(\widetilde{Q}(w)+\veps\Id)$ generates an analytic semigroup on some sector $\Sigma_{\alpha'}$ with $\alpha'>0$.
\end{proof}

\begin{lemma}\label{le:C1}
Suppose Assumptions \eqref{Q}, \hyl{smooth} and \hyl{ass:0} hold, and let 
 $\varphi\colon[-\delta,0]\to X$ be continuously differentiable. Further assume $\varphi$ satisfies the boundary condition \eqref{eq:C1} and $\varphi([-\delta,0])\subset W$. Then there exists a unique solution $v:=v_\varphi\colon[-\delta,\epsilon]\to X$ to \eqref{eq:delay}, and we have $v\in \C^1([-\delta,\epsilon],X)$ and $v(t)\in D$ for all $t\in[0,\epsilon]$.
\end{lemma}

\begin{proof}
Extend $\varphi$ in a continuously differentiable way to $[-\delta,\epsilon]$ by, say, $\varphi(\sigma):=\varphi(0)+\sigma\dot\varphi_{-}(0)$ for $\sigma\geq0$.\\
Let $J:=[0,\epsilon]$ and $A(t):=Q(F(\t\varphi))$. Note that $A(t)$ is then actually independent of the extension $\varphi|_{[0,\epsilon]}$ by Assumption \hyl{ass:0}. By Assumption \eqref{Q}, $D(A(t))=D$ for all $t\in J$, and
$
A(t)=Q_0+\widetilde{Q}(F(\t \varphi))
$. By our choice of extension for $\varphi$, $\t\varphi$ is continuously differentiable as a map $t\mapsto \t \varphi$ from $[0,\epsilon]$ to $\C([-\delta,0],X)$ (with derivative $\t{\dot\varphi}$). By Assumptions \hyl{smooth}, \hyl{ass:0}, we then have that $A(\cdot)x\in \C^1([0,\epsilon],X)$ for all $x\in D$. Since also $\varphi(0)\in D$, this allows us to apply Corollary \ref{cor:wellposed} to obtain that there exists a unique $v\colon J\to X$ with  $v(0)=\varphi(0)$ that solves the corresponding problem \eqref{eq:ncp}, Further, $v$ is continuously differentiable on $J$, and $v(t)\in D$ for all $t\in J$.  
By \eqref{eq:ncp} and \eqref{eq:C1}, we also have $v_{+}'(0)=\varphi_{-}'(0)$, so extending $v$ to $[-\delta,0]$ as $v=\varphi$ we preserve continuous differentiability.
\end{proof}

These results will allow us to apply Theorem \ref{thm:ostermann} in our setting. However, to obtain second-order convergence of the Magnus-type integrator, we need stronger smoothness properties than what we have shown above, as, among other things, the bounds on the RHS of \eqref{eq:rhs} are not automatically finite.

\begin{lemma}\label{le:C2}
Under the conditions of Lemma \ref{le:C1}, whenever
the solution $v:=v_\varphi$ to \eqref{eq:delay} is twice differentiable on $[0,\epsilon]$, we have $v(t)\in D(Q_0^2)\eqqcolon D_1$ for all $t\in [0,\epsilon]$,  $Q_0v(t)$ is differentiable on $[0,\epsilon]$ with derivative $Q_0\dot{v}(t)$, and the function $\ddot{v}(t)$ is continuous on $[0,\epsilon]$ if and only if $(Q(v(t-\delta)))^2v(t)$ (or, equivalently, $Q_0\dot{v}(t)$) is continuous.\\
On the other hand, whenever Assumptions \eqref{Q}, \hyl{smooth}, \eqref{dissip}, \eqref{ass:1} and \eqref{ass:4} hold,
we have $v\in \C^2([-\delta,\epsilon],X)\cap \C^1([-\delta,\epsilon],D)$.
\end{lemma}
\begin{proof}

By Lemma \ref{le:C1} the solution $v$ is continuously differentiable on $[-\delta,\epsilon]$, $v(t)\in D$ for all $t\in[0,\epsilon]$, and \eqref{eq:delay} is satisfied. Now for $t\in[0,\epsilon]$, we have for small enough $h$ (for $t=0$ and $t=\epsilon$ only with $h>0$ and $h<0$, respectively) that
\begin{align}\label{eq:quotient}
\nonumber \frac{\dot{v}(t+h)-\dot{v}(t)}{h}=&\frac{Q(F(\prescript{}{t+h}v))v(t+h)-Q(F(\t v))v(t)}{h}\\
\nonumber =&(Q(F(\prescript{}{t+h}v))-Q(F(\t v)))\frac{v(t+h)-v(t)}{h}\\
\nonumber &+Q_0\frac{v(t+h)-v(t)}{h}\\
\nonumber &+\widetilde{Q}(F(\t v))\frac{v(t+h)-v(t)}{h}\\
&+\frac{Q(F(\prescript{}{t+h}v))-Q(F(\t v))}{h}v(t).
\end{align}
In the first term of the last expression, $Q(F(\prescript{}{t+h}v))-Q(F(\t v))=\widetilde Q(F(\prescript{}{t+h}v))-\widetilde Q(F(\t v))\in\Ell(X)$ tends to zero in operator norm, and $v$ being a classical solution implies that $\frac{v(t+h)-v(t)}{h}$ converges as well, so the first term tends to 0 as $h\to 0$. Similarly, the fourth term converges to
\[
\left(\ds \widetilde{Q}(F(\prescript{}{s}v))|_{s=t}\right) v(t)=\left(\widetilde{Q}'(F(\prescript{}{t}v))F'(\prescript{}{t}v)\t{\dot v}\right)v(t),
\] and the third term to $\widetilde{Q}(F(\t v))\dot{v}(t)$. Note that these are all continuous in $t$. Finally, since the left-hand side has a limit as $h\to 0$, by closedness of $Q_0$ one has $\lim_{h\to 0}Q_0\frac{v(t+h)-v(t)}{h}=Q_0\dot{v}(t)$.\\
Also $(Q(F(\t v)))^2v(t)=Q_0\dot{v}(t)+\widetilde{Q}(F(\t v))\dot{v}(t)$, where $\widetilde{Q}(F(\t v))\dot{v}(t)$ is continuous, so continuity of $\ddot{v}(t)$ and $(Q(F(\t v)))^2v(t)$ on $[0,\epsilon]$ are both equivalent to continuity on $[0,\epsilon]$ of $Q_0\dot{v}(t)$.

For the last claim, in light of \citet[Prop. II.5.2]{Engel-Nagel}, Assumptions \eqref{dissip} and \eqref{ass:1} imply that Lemmas \ref{le:contraction} and \ref{le:C1} actually apply also when replacing the space $X$ by $D$, and each generator by its restriction. In addition, $D$ being a subspace of $X$ with a stronger norm implies that whenever $\dot\varphi(t)$ ($t\in[-\delta,0]$) or $\dot{v}(t)$ ($t\geq0$) exists in both $X$ and $D$, the derivatives actually coincide.
Hence $v\in \C^1([-\delta,\epsilon],D)$, which implies that $Q_0v\in \C^1([-\delta,\epsilon],X)$, meaning that all four terms in equation \eqref{eq:quotient} above have $\|\cdot\|_X$-continuous limits, i.e., $\ddot{v}(t)\in \C([0,\epsilon],X)$.
Finally, the condition \eqref{eq:C2} on $\ddot\varphi_{-}(0)$ serves to match the left and right second derivatives of the solution $v\colon[-\delta,\epsilon]\to X$ at 0, so $v\in \C^2([-\delta,\epsilon],X)$.
\end{proof}

As a consequence of the above, applied to the appropriate initial history functions, we will be able to guarantee finiteness of the RHS in \eqref{eq:rhs}. More specifically, we have the following result.

\begin{proposition}\label{prop:g}
Suppose Assumptions \eqref{Q}, \eqref{smooth2}, \eqref{dissip}, \eqref{ass:1} and \eqref{ass:4} hold, and $\varphi\in \C^2([-\delta,0],X)\cap \C^1([-\delta,0],D)$ satisfies the boundary conditions \eqref{eq:C1} and \eqref{eq:C2}. Let $v$ denote the unique solution to \eqref{eq:delay} guaranteed by Lemma \ref{le:C1}. Then for any $\tau>0$ and $s\in[0,\epsilon-\tau]$, the functions $\gamma_s(t):=\gamma_{\varphi,s}:~[s,s+\tau]\to X$ defined as
\begin{equation*}
\gamma_s(t):=\big(Q(F(\prescript{}{s}v))-Q(F(\prescript{}{s+\tau/2}v))\big)v(t)
\end{equation*}
satisfy $\gamma_s(t)\in D$ for all $t\in[s,s+\tau]$, and $\gamma_s\in \C^2([s,s+\tau],X)\cap \C^1([s,s+\tau],D)$.
\end{proposition}

\begin{proof}
We have
\begin{align*}
\big(Q(F(\prescript{}{s}v))-Q(F(\prescript{}{s+\tau/2} v))\big)v(t)&=\big(\Q(F(\prescript{}{s}v))-\Q(F(\prescript{}{s+\tau/2} v))\big)v(t)\\
&=\left(\Q(F(\prescript{}{s}v))\right)v(t)-\left(\Q(F(\prescript{}{s+\tau/2} v))\right)v(t).
\end{align*}
By Lemma \ref{le:C2} we have $v\in \C^2([-\delta,\epsilon],X)\cap \C^1([-\delta,\epsilon],D)$, and together with Assumptions \eqref{smooth2}, \eqref{dissip} and \eqref{ass:1}, we are done.
\end{proof}

We note here that if $W=X$, all of the above results would automatically extend to the interval $[-\delta,\infty)$ by iteration, as the solution obviously stays in $X$. However, in many applications some of the assumptions may only be known for certain special subsets $W$ of $X$. A typical hurdle would be that no choice of $c$ will render all of the operators $\widehat{Q}(w):=Q(w)-c\Id$ generators of contraction semigroups. What may happen instead is that $X$ is some function space, and the Assumptions are only valid for the positive cone $W:=X_+$, or more generally, for sets of functions $W$ that are uniformly bounded below, and this is exactly the case in our example from Section \ref{sec:exm}. \medskip

In such cases, the hope is that the set $W$ is in some way ``closed'' or ``invariant'' under both \eqref{eq:delay} and the numerical method so the argument can be iterated with the Assumptions still satisfied. This double invariance is not unheard of, as in the case of PDE's, positivity of the solution semigroup and positivity-preserving numerical methods are a well-studied area.\\
For the invariance of $W$ under \eqref{eq:delay}, we shall actually make use of the convergence of the Magnus-type integrator and its $W$-preserving property, in a way somewhat reminiscent of how positivity of operator semigroups is shown via the Post--Widder Inversion Formula \citep[cf.][Corollary III.5.5]{Engel-Nagel}.

In what follows we will use the following notations for any $T>0$ for which the solution $u=u_\varphi$ is known to exist as a (twice) differentiable function $[-\delta,T]\to X$:
\begin{enumA}
\item $M_{u,T}:=\|u\|_{\C([-\delta,T],X)}$,
\item $L_{u,T}$ and $L_Q$ denote the Lipschitz constants of the functions $u$ and $\Q|_{u([-\delta,T])}$, respectively, that is, $\|u(t)-u(s)\|\le L_{u,T}|t-s|$ for all $T\geq t,s\ge -\delta$, and $\|\Q(w_1)-\Q(w_2)\|\le L_Q\|w_1-w_2\|$ for all $w_1,w_2\in u([-\delta,T])$.
\end{enumA}
We state now our main result.
\begin{theorem}\label{thm:main} Let $X$ be a Banach space, $m\geq 0$ an arbitrary integer, $\delta>0$ and $\epsilon\in(0,\delta]$.
Suppose further that Assumptions \eqref{Q}--\eqref{ass:4} are satisfied.
Then all of the following hold:
\begin{enumA}
\item\label{W1} the unique solution $u=u_\varphi$ to \eqref{eq:delay} exists for all $t\in[-\delta,m\epsilon]$ and $u\in \C^2([-\delta,m\epsilon],X)\cap \C^1([-\delta,m\epsilon],D)$,
\item\label{W2} $u(t)\in W$ for all $t\in[-\delta,m\epsilon]$,
\item\label{W3} the Magnus-type integrator \eqref{eq:magnus} applied to \eqref{eq:delay} never leaves the set $W$ on $[-\delta,m\epsilon]$, i.e., for $(n+1)\tau\leq m\epsilon$, we have $u^{(\tau)}_{n+1}\in W$,
\item\label{W4} the Magnus-type integrator is convergent of second order on $[-\delta,m\epsilon]$, i.e., there exists a constant $\alpha_m\geq0$ independent of $n$ and $\tau$ such that
\begin{equation}\label{eq:rec_bound}
\|u((n+1)\tau)-u_{n+1}^{(\tau)}\|\le \alpha_m\tau^2
\end{equation}
holds for all $n\in\NN$ with $(n+1)\tau\in[0,m\epsilon]$.
\end{enumA}
\end{theorem}
\begin{proof}
We shall proceed by induction on $m$. Indeed, let us assume that assertions \eqref{W1}--\eqref{W4} are valid for some integer $m\geq0$ and all $\tau>0$ with $\delta/\tau\in\mb{N}$. This is trivially true for $m=0$.\\
Now, when $t\leq (m+1)\epsilon$, the generator $A(t):=Q(F(\t u))$ only depends on the $\C^1$ function $u|_{[-\delta,m\epsilon]}$, and by the inductive assumption this runs in $W$.\\
On the one hand, this means that the Magnus-type integrator \eqref{eq:magnus} will indeed stay in $W$ at least until time $(m+1)\epsilon$, covering the inductive step for \eqref{W3}.\\
On the other hand, letting $\phi:=\prescript{}{m\epsilon}u$, this also means that $\phi\in\C^2([-\delta,0],X)\cap \C^1([-\delta,0],D)$ and $\phi$ satisfies the boundary conditions \eqref{eq:C1} and \eqref{eq:C2}. Hence Lemma \ref{le:C2} can be applied to the equation (QDE$_\phi$), implying that the solution $u$ can be uniquely extended to $[m\epsilon, (m+1)\epsilon]$ as $u|_{[m\epsilon, (m+1)\epsilon]}:=v_\phi$, and then $u\in\C^2([-\delta,(m+1)\epsilon],X)\cap\C^1([-\delta,(m+1)\epsilon],D)$, completing the inductive step for \eqref{W1}.\\
We now claim that the conditions of Theorem \ref{thm:ostermann} are satisfied for $T:=(m+1)\epsilon$ and $A(t):=Q(F(\t u))$, and the bound on the RHS of \eqref{eq:rhs} is finite. Indeed, the norm on $D$ is defined as the graph norm of $Q_0$, and bounded perturbations always lead to an equivalent graph norm, hence the graph norms of $A(t)$ are all equivalent to $\|\cdot\|_D$. The uniform sectoriality property follows from Lemma \ref{le:sector}. Next note that $A(t)=Q_0+\widetilde{Q}(F(\t v))$, and by \eqref{W1} and Assumptions \eqref{smooth2} and \eqref{ass:1} we have $\widetilde{Q}(F(\prescript{}{\cdot} u))\in \C^1([-\delta,(m+1)\epsilon],\Ell(X))$. Since the norm on $\Ell(X)$ dominates the norm on $\Ell(D,X)$, we a fortiori have $\widetilde{Q}(F(\prescript{}{\cdot}v))\in \C^1([-\delta,\epsilon],\Ell(D,X))$. Concerning $g_n$ in \eqref{eq:g}, note that for any $n$ such that $(n+1)\tau\in(\epsilon,(m+1)\epsilon]$, we have $g_n=\gamma_{\prescript{}{(n+1)\tau-\epsilon}u,\epsilon-\tau}$, which by Proposition \ref{prop:g} is in $\C^2([(n+1)\tau-\epsilon,(n+1)+\tau],X)\cap\C^1([(n+1)\tau-\epsilon,(n+1)+\tau],D)$, hence the RHS in \eqref{eq:rhs} is indeed finite.\\
Knowing this we are ready to show that the Magnus-type integrator maintains its second-order convergence up to $(m+1)\epsilon$. \medskip

Let us denote the global error of the method by
\begin{equation*}
\veps_n=\left\{\begin{array}{cl}
0 &\:\text{for } n=-2N,\dots,0, \\
\|u(n\tau)-u_n^{(\tau)}\| &\:\text{for } n=1,2,\dots.
\end{array}\right.
\end{equation*}
Hence, we need to estimate the term $\veps_{n+1}$ for $n\in\NN$ with $(n+1)\tau\in[0,(m+1)\epsilon]$.
The triangle inequality implies
\begin{equation}\label{eq:tri1}
\veps_{n+1}=\|u((n+1)\tau)-u_{n+1}^{(\tau)}\|\le\underbrace{\|u((n+1)\tau)-\widehat u_{n+1}^{(\tau)}\|}_{\hyt{\text{(T1)}}{t1}}+\underbrace{\|\widehat u_{n+1}^{(\tau)}-u_{n+1}^{(\tau)}\|}_{\hyt{\text{(T2)}}{t2}},
\end{equation}
where $\widehat u_{n+1}^{(\tau)}$ is defined in \eqref{eq:magnusA} with $A(t)=Q(F(\t u))$, i.e., with the exact solution $u$, as
\begin{equation*}
\widehat u_{n+1}^{(\tau)}=\ee^{\tau Q(F(\prescript{}{(n+1/2)\tau}u))}\widehat u_n^{(\tau)}.
\end{equation*}
The first term in \eqref{eq:tri1} is the global error of method \eqref{eq:magnusA} being of second order by Theorem \ref{thm:ostermann}:
\begin{equation}\label{eq:term1}
\hyl{t1}=\|u((n+1)\tau)-\widehat u_{n+1}^{(\tau)}\|\le C_0\tau^2
\end{equation}
with the finite constant $C_{0,m}:=C\left(\left\|g'|_{[-\delta,m\epsilon]}\right\|_{D,\infty}+\left\|g''|_{[-\delta,m\epsilon]}\right\|_{X,\infty}\right)$. \medskip

For bounding \hyl{t2}, we write the difference as a telescopic sum of differences where only one term of the product changes at a time. By Assumption \eqref{ass:3} and \eqref{W1} we may apply Lemma \ref{le:contraction}, and using $n\tau<(m+1)\epsilon$ write
\begin{align}
\nonumber\hyl{t2}=&\|\widehat u_{n+1}^{(\tau)}-u_{n+1}^{(\tau)}\|
= \Big\|\prod\limits_{k=0}^n\ee^{\tau Q(F(\prescript{}{(k+1/2)\tau }u))}\varphi(0)-\prod\limits_{k=0}^n\ee^{\tau 
Q\left(
\sum_{\ell=0}^{\lfloor\frac{\delta-\epsilon}{\tau}\rfloor}
\kappa_{\ell,\tau} F_{\ell,\tau}\left(u_{k+\ell+1/2}^{(\tau)}\right)
\right)}\varphi(0)\Big\|
\\
\nonumber\le&\left\|\sum\limits_{j=0}^n\Big(\prod\limits_{k=j+1}^n\ee^{\tau Q(F(\prescript{}{(k+1/2)\tau }u))}\Big)\cdot\left(\ee^{\tau Q(F(\prescript{}{(j+1/2)\tau }u))}-\ee^{\tau Q\left(
\sum_{\ell=0}^{\lfloor\frac{\delta-\epsilon}{\tau}\rfloor}
\kappa_{\ell,\tau} F_{\ell,\tau}\left(u_{j+\ell+1/2}^{(\tau)}\right)
\right)}\right)\right.\\
\nonumber&\left.\cdot\left(\prod\limits_{k=0}^{j-1}\ee^{\tau Q\left(
\sum_{\ell=0}^{\lfloor\frac{\delta-\epsilon}{\tau}\rfloor}
\kappa_{\ell,\tau} F_{\ell,\tau}\left(u_{k+\ell+1/2}^{(\tau)}\right)
\right)}\right)\varphi(0)\right\|
\\
\nonumber\le&\sum\limits_{j=0}^n\Big(\prod\limits_{k=j+1}^n\big\|\ee^{\tau Q(F(\prescript{}{(k+1/2)\tau }u))}\big\|\Big)\cdot\left\|\ee^{\tau Q(F(\prescript{}{(j+1/2)\tau }u))}-\ee^{\tau Q\left(
\sum_{\ell=0}^{\lfloor\frac{\delta-\epsilon}{\tau}\rfloor}
\kappa_{\ell,\tau} F_{\ell,\tau}\left(u_{j+\ell+1/2}^{(\tau)}\right)
\right)}\right\|\\
\nonumber&\cdot\left(\prod\limits_{k=0}^{j-1}\left\|\ee^{\tau Q\left(
\sum_{\ell=0}^{\lfloor\frac{\delta-\epsilon}{\tau}\rfloor}
\kappa_{\ell,\tau} F_{\ell,\tau}\left(u_{k+\ell+1/2}^{(\tau)}\right)
\right)}\right\|\|\varphi(0)\|\right)
\\
\le& \ee^{c(m+1)\epsilon}\sum\limits_{j=0}^n\underbrace{\left\|\ee^{\tau Q(F(\prescript{}{(j+1/2)\tau }u))}-\ee^{\tau Q\left(
\sum_{\ell=0}^{\lfloor\frac{\delta-\epsilon}{\tau}\rfloor}
\kappa_{\ell,\tau} F_{\ell,\tau}\left(u_{j+\ell+1/2}^{(\tau)}\right)
\right)}\right\|}
_{\hyt{\text{(T3$j$)}}{t3j}}
\|\varphi(0)\|\label{eq:term2}.
\end{align}
Now we use the variation of constants formula and Lipschitz continuity of $\widetilde{Q}$ to obtain (recall $\tau\leq(m+1)\epsilon$)
\begin{align}
\nonumber \hyl{t3j}=&
\left\| \int_0^\tau \ee^{(\tau-s)Q\left(\sum_{\ell=0}^{\lfloor\frac{\delta-\epsilon}{\tau}\rfloor}
\kappa_{\ell,\tau} F_{\ell,\tau}\left(u_{j+\ell+1/2}^{(\tau)}\right)\right)}\right.\\
&\cdot\left.
\left(
Q(F(\prescript{}{(j+1/2)\tau }u))-
Q\left(
          \sum_{\ell=0}^{\lfloor\frac{\delta-\epsilon}{\tau}\rfloor}
          \kappa_{\ell,\tau} F_{\ell,\tau}\left(u_{j+\ell+1/2}^{(\tau)}\right)
    \right)
\right)
\ee^{(\tau-s)
Q(F(\prescript{}{(j+1/2)\tau }u))}\dd s 
\right\| \nonumber 
\\
\le & \tau \ee^{2c(m+1)\epsilon}L_Q\underbrace{
\left\|F(\prescript{}{(j+1/2)\tau }u)-\sum_{\ell=0}^{\lfloor\frac{\delta-\epsilon}{\tau}\rfloor}
          \kappa_{\ell,\tau} F_{\ell,\tau}\left(u_{j+\ell+1/2}^{(\tau)}\right)
\right\|}_{\hyt{\text{(T4$j$)}}{t4j}}\label{eq:term3}.
\end{align}
Here we will have to insert an approximating term involving $F_\tau$ defined in \eqref{eq:kvadra} and use Assumption \eqref{ass:2}. We have
\begin{align}\label{eq:term4*}
\hyl{t4j}\leq 
\underbrace{\left\|
F(\prescript{}{(j+1/2)\tau }u)-
F_\tau(\prescript{}{(j+1/2)\tau }u)
\right\|}
_{\hyt{\text{(T4a$j$)}}{t4aj}}+
\underbrace{\left\|
F_\tau(\prescript{}{(j+1/2)\tau }u)-
\sum_{\ell=0}^{\lfloor\frac{\delta-\epsilon}{\tau}\rfloor}
          \kappa_{\ell,\tau} F_{\ell,\tau}\left(u_{j+\ell+1/2}^{(\tau)}\right)
          \right\|}
_{\hyt{\text{(T4b$j$)}}{t4bj}},
\end{align}
and by Assumption \eqref{ass:2} and inequality \eqref{eq:kvadr_err}, we have
\begin{equation}\label{eq:term4a}
\hyl{t4aj}\leq  \ms{C}\tau^2\|u\|_{\C^2([-\delta,m\epsilon],X)}.
\end{equation}
Now we turn our attention to the second term, and obtain by the triangle inequality and the uniform Lipschitz continuity of the $F_{*,\tau}$'s that
\begin{align}
\hyl{t4bj}&\leq
\sum_{\ell=0}^{\lfloor\frac{\delta-\epsilon}{\tau}\rfloor}
\left|\kappa_{\ell,\tau}\right|L_F
\left\|
u((j+\ell+1/2)\tau-\delta)-
u_{j+\ell+1/2}^{(\tau)}
          \right\|
\nonumber\\
&=\sum_{\ell=j}^{ j+\lfloor\frac{\delta-\epsilon}{\tau}\rfloor}
\left|\kappa_{\ell-j,\tau}\right|L_F
\underbrace{\left\|
u((\ell+1/2)\tau-\delta)-
u_{\ell+1/2}^{(\tau)}
          \right\|}
_{\hyt{\text{(T5$\ell$)}}{t5l}}.\label{eq:term4b}
\end{align}
Since $(n+1)\tau\leq(m+1)\epsilon$, we have up to now shown
\begin{align}\label{eq:est1}
\hyl{t2}\leq \tau C_{1,m}\left(
\ms{C}\tau(m+1)\epsilon\|u\|_{\C^2([-\delta,m\epsilon],X)}+
\sum_{j=0}^n 
\sum_{\ell=j}^{ j+\lfloor\frac{\delta-\epsilon}{\tau}\rfloor}
\left|\kappa_{\ell-j,\tau}\right|L_F
\hyl{t5l}
\right),
\end{align}
where $C_{1,m}:=\ee^{3c(m+1)\epsilon}L_Q\|\varphi(0)\|$.
So now we have to bound the norm \hyl{t5l}. Since both terms in the difference take the same value $\varphi((\ell+1/2)\tau-\delta)$, the norm \hyl{t5l} is equal to zero for $\ell=0,\dots,N-1$, hence, we only need to consider the indices $\ell\ge N$. 

By the inductive hypothesis there exists an $\alpha_m$ independent of $\tau$ such that for any index $q\geq -2N$ with $q\tau\leq m\epsilon$ we have $\veps_q\leq\alpha_m\tau^2$.
Recall that by the definitions of the evolution family $\U(\cdot,\cdot)$ and of our Magnus-type integrator
\begin{align*}
u((\ell+1/2)\tau-\delta)&=\U((\ell+1/2)\tau-\delta,\ell\tau-\delta)u(\ell\tau-\delta) \quad \text{and}\\
u_{\ell+1/2}^{(\tau)}&=\ee^{\frac\tau2Q\left(
\sum_{k=0}^{\lfloor\frac{\delta-\epsilon}{\tau}\rfloor}
\kappa_{k,\tau} F_{k,\tau}\left(u_{\ell-2N+k}^{(\tau)}\right)
\right)}u_{\ell-N}^{(\tau)}
\end{align*}
for all $\ell\ge N$, hence inserting the term $\ee^{\frac\tau2Q\left(
\sum_{k=0}^{\lfloor\frac{\delta-\epsilon}{\tau}\rfloor}
\kappa_{k,\tau} F_{k,\tau}\left(u_{\ell-2N+k}^{(\tau)}\right)
\right)}u(\ell\tau-\delta)$ and applying the triangle inequality leads to
\begin{align*}\
\hyl{t5l} &\le \underbrace{\left\|\U((\ell+1/2)\tau-\delta,\ell\tau-\delta)-\ee^{\frac\tau2Q\left(
\sum_{k=0}^{\lfloor\frac{\delta-\epsilon}{\tau}\rfloor}
\kappa_{k,\tau} F_{k,\tau}\left(u_{\ell-2N+k}^{(\tau)}\right)
\right) }\right\|}_{\hyt{\text{(T6$\ell$)}}{t6l}}\underbrace{\|u(\ell\tau-\delta)\|}_{\le M_{u,m\epsilon}} \\
&+\underbrace{\left\|\ee^{\frac\tau2Q\left(
\sum_{k=0}^{\lfloor\frac{\delta-\epsilon}{\tau}\rfloor}
\kappa_{k,\tau} F_{k,\tau}\left(u_{\ell-2N+k}^{(\tau)}\right)
\right)}\right\|}_{\le \ee^{c(m+1)\epsilon/2}}\underbrace{\|u(\ell\tau-\delta)-u_{\ell-N}^{(\tau)}\|}_{=\veps_{\ell-N}}\\
&\leq M_{u,m\epsilon}\hyl{t6l}+C_{2,m}\tau^2
\end{align*}
where $C_{2,m}:=\alpha_m\ee^{c(m+1)\epsilon/2}$.

Approximating the evolution family using the midpoint rule we have
\begin{align*}
\hyl{t6l}\leq&
\underbrace{\left\|\U((\ell+1/2)\tau-\delta,\ell\tau-\delta)-\ee^{\frac\tau2Q\left(
F\left(
\prescript{}{(\ell+1/4)\tau-\delta}u
\right)
\right) }\right\|}_{\hyt{\text{(T6a$\ell$)}}{t6al}}\\
{}&+
\underbrace{\left\|\ee^{\frac\tau2Q\left(
F\left(
\prescript{}{(\ell+1/4)\tau-\delta}u
\right)\right) }
-\ee^{\frac\tau2Q\left(
\sum_{k=0}^{\lfloor\frac{\delta-\epsilon}{\tau}\rfloor}
\kappa_{k,\tau} F_{k,\tau}\left(u_{\ell-2N+k}^{(\tau)}\right)
\right)} \right\|}_{\hyt{\text{(T6b$\ell$)}}{t6bl}}.
\end{align*}

Since we for all $t,s\in[0,m\epsilon]$ have
\[
\|Q(F(\prescript{}{t}u))-Q(F(\prescript{}{s}u))\|\leq L_QL_F\|\prescript{}{t}u-\prescript{}{s}u\|_\infty\leq L_QL_FL_{u,m\epsilon}|t-s|,
\]
where $L_F$ is the Lipschitz constant of the continuous function $F\in\C^2\left(\C([-\delta,0],X),X\right)$, Theorem \ref{thm:batkai} and \eqref{eq:magnuscons} imply
\begin{align*}
\hyl{t6al}\leq L_QL_FL_{u,m\epsilon}e^{c(m+1)\epsilon/2}\left(\frac\tau 2\right)^2.
\end{align*}
Furthermore the variation of constants formula and the uniform quasi-contractivity of the semigroups involved guaranteed by \eqref{W1} and Assumption \eqref{ass:3} yield
\begin{align*}
\hyl{t6bl}\leq\frac\tau2 \ee^{c(m+1)\epsilon}\underbrace{
\left\|
Q\left(
F\left(
\prescript{}{(\ell+1/4)\tau-\delta}u
\right)\right) -
Q\left(
\sum_{k=0}^{\lfloor\frac{\delta-\epsilon}{\tau}\rfloor}
\kappa_{k,\tau} F_{k,\tau}\left(u_{\ell-2N+k}^{(\tau)}\right)
\right)
\right\|
}_{\hyt{\text{(T7$\ell$)}}{t7l}}.
\end{align*}
Finally, we have using the bound \eqref{eq:kvadr_err} 
\begin{align*}
\hyl{t7l}\leq&L_Q\left\|
F\left(
\prescript{}{(\ell+1/4)\tau-\delta}u
\right)-
\sum_{k=0}^{\lfloor\frac{\delta-\epsilon}{\tau}\rfloor}
\kappa_{k,\tau} F_{k,\tau}\left(u_{\ell-2N+k}^{(\tau)}\right)
\right\|\\
\leq&L_Q\left(\left\|
F\left(
\prescript{}{(\ell+1/4)\tau-\delta}u
\right)-
F\left(
\prescript{}{\ell\tau-\delta}u
\right)
\right\|
+\left\|
F\left(
\prescript{}{\ell\tau-\delta}u
\right)-
\sum_{k=0}^{\lfloor\frac{\delta-\epsilon}{\tau}\rfloor}
\kappa_{k,\tau} F_{k,\tau}\left(u_{\ell-2N+k}^{(\tau)}\right)
\right\|\right)\\
\leq& L_QL_FL_{u,m\epsilon}\frac\tau4\\
{}&+L_Q
\left(
\left\|F\left(
\prescript{}{\ell\tau-\delta}u
\right)-F_\tau\left(
\prescript{}{\ell\tau-\delta}u
\right)\right\|+
\left\|
F_\tau\left(
\prescript{}{\ell\tau-\delta}u
\right)-
\sum_{k=0}^{\lfloor\frac{\delta-\epsilon}{\tau}\rfloor}
\kappa_{k,\tau} F_{k,\tau}\left(u_{\ell-2N+k}^{(\tau)}\right)\right\|
\right)\\
\leq&L_QL_FL_{u,m\epsilon}\frac\tau4+L_Q\ms{C}\|u\|_{\C^2([-\delta,m\epsilon])}\tau^2+L_Q
\sum_{k=0}^{\lfloor\frac{\delta-\epsilon}{\tau}\rfloor}
\left|\kappa_{k,\tau}\right|\underbrace{\|u(\ell\tau-\delta+k\tau)-u_{\ell-2N+k}^{(\tau)}\|}_{\veps_{\ell+k-2N}}.
\end{align*}

 Now note that $\lfloor\frac{\delta-\epsilon}{\tau}\rfloor-N\leq -\frac\epsilon\tau<0$, so all indices for the $\epsilon$'s appearing above are between $-2N$ and $n-\lfloor\epsilon/\tau \rfloor$. Also, $\tau\leq \delta$.
Thus, using \eqref{eq:weights}, we obtain 
\begin{align*}
\hyl{t7l}\leq L_Q\left(L_fL_{u,m\epsilon}/4+ \ms{C}\|u\|_{\C^2([-\delta,m\epsilon])}\delta+L\alpha_m\delta\right)\tau,
\end{align*}
and so
\begin{align*}
\hyl{t6l}\leq C_{3,m}\tau^2
\end{align*}
where $C_{3,m}:=\ee^{c(m+1)\epsilon/2}L_QL_FL_{u,m\epsilon}/4+\ee^{c(m+1)\epsilon}L_Q\left(L_fL_{u,m\epsilon}/4+ \ms{C}\|u\|_{\C^2([-\delta,m\epsilon])}\delta+L\alpha_m\delta\right)/2$.

Hence
\begin{align}\label{eq:t5l}
\hyl{t5l}\leq 
M_{u,m\epsilon}\hyl{t6l}+C_{2,m}\tau^2\leq C_{4,m}\tau^2
\end{align}
for all $\ell\geq N$ where $C_{4,m}:=M_{u,m\epsilon}C_{3,m}+C_{2,m}$, and so also for all $\ell\geq0$.
By assumption $(n+1)\tau\leq(m+1)\epsilon$, hence
\begin{equation*}
\hyl{t2}\leq \tau C_{1,m}\left(
\ms{C}\tau(m+1)\epsilon\|u\|_{\C^2([-\delta,m\epsilon],X)}+
\sum_{j=0}^n 
\sum_{\ell=j}^{ j+\lfloor\frac{\delta-\epsilon}{\tau}\rfloor}
\left|\kappa_{\ell-j,\tau}\right|L_F
\hyl{t5l}
\right)
\leq C_{5,m}\tau^2
\end{equation*}
with $C_{5,m}:=C_{1,m}(m+1)\epsilon\left(\ms{C}\|u\|_{\C^2([-\delta,m\epsilon],X)}+LL_FC_{4,m}\right)$.

Thus, combining this with inequality \eqref{eq:term1} and substituting into inequality \eqref{eq:tri1}, we obtain
\begin{align*}
\veps_{n+1}\leq (C_{0,m}+C_{5,m})\tau^2,
\end{align*}
and setting
\[
\alpha_{m+1}:=C_{0,m}+C_{5,m}
\]
concludes the inductive step for \eqref{W4}.

Finally, to see that claim \eqref{W2} is also true, we choose $N_\ell:=2^\ell$, hence, $\tau_\ell=\delta/2^\ell$. Then by \eqref{W4} for any $q/2^z\in[0,(m+1)\epsilon/\delta]$ ($q,z\in\mb{N}$) we have that
\begin{equation*}
\lim_{\ell\to\infty} u_{q2^{\ell-z}}^{(\tau_\ell)}=u(\tfrac{q}{2^z}\delta)
\end{equation*}
with the left-hand side consisting of points in $W$ only. Since $W$ is closed, $u(\frac{q}{2^z}\delta)\in W$ for any $\delta q/2^z\in[0,(m+1)\epsilon]$ ($q,z\in\mb{N}$), but by continuity of the solution this then implies $u(t)\in W$ for all $t\in[0,(m+1)\epsilon]$.
\end{proof}

\begin{remark} If $\varphi$ is only given at the grid points, and hence the values $\varphi((\ell+1/2)\tau-\delta)$ are not known (or for any reason the use of off-grid exact values in the method is undesirable),  we may amend the method and approximate these values by the average $(\varphi(\ell\tau-\delta)+\varphi((\ell+1)\tau-\delta))/2$ for indices $\ell=0,\dots,N-1$. The error of this approximation can be bounded above by $C_6\tau^2$ with $C_6:=\|\varphi''\|_\infty/4$. This would lead to an upper bound on \hyl{t5l} of unchanged order of magnitude $O(\tau^2)$, with the new constant $C_{4,m}:=\max\left\{M_{u,m\epsilon}C_{3,m}+C_{2,m},C_6\right\}$ in \eqref{eq:t5l}.
\end{remark}

At this point it is natural to ask whether higher order convergence may be achieved by applying a higher order truncation of the Magnus expansion. However, to our knowledge not even the building blocks Theorems \ref{thm:ostermann} and \ref{thm:batkai} have higher order variants available yet, so these are all open questions to be further investigated.

\section{Application to an epidemic model}
\label{sec:exm}

As an illustrative example we treat an epidemic model which is based on the classical SIR model introduced in \citet{Kermack-McKendrick-1927} but takes into account the effect of the vaccination, the space-dependency of infection \citep[cf.][]{Kendall}, and the random movement of individuals as well. And most importantly we include the latent period which leads to a delay equation (see, e.g., in \citet{He-Tsai}, \citet{Huang-Takeuchi}, \citet{Xu}). \medskip

Let $\Omega\subset\RR^2$ be the space domain, $\delta>0$ the latent period, and for all time values $t\geq -\delta$ let $S(t),I(t),R(t)\colon\Omega\to\RR$ denote the spatial distribution of susceptible, infected, and recovered individuals within the total population, respectively. We assume that each of these functions lies in $Y:=L^2(\Omega)$, and our state space will be the Hilbert space $X:=Y^3$ with the norm 
\begin{equation*}
\|(x_1,x_2,x_3)\|_X^2:=\|x_1\|_2^2+\|x_2\|_2^2+\|x_3\|_2^2.
\end{equation*}
We will also use the norm 
\begin{equation*}
\|(x_1,x_2,x_3)\|_1:=\|x_1\|_1+\|x_2\|_1+\|x_3\|_1.
\end{equation*}
Note that for any $h\in L^2(\Omega)$ we have $\|h\|_1\leq \|h\|_2\cdot\lambda(\Omega)^{1/2}$, hence for any $x\in X$, we have $\|x\|_1\leq \|x\|_X\sqrt{3\lambda(\Omega)}$, where $\lambda(\Omega)$ denotes the Lebesgue measure of $\Omega$. \medskip

The temporal change of $S$, $I$, $R$ depends on various phenomena, from which we first consider the infection-related ones. The number of susceptible individuals decreases because they are in contact with infected people and get infected. More precisely, the actual change in the number of susceptibles depends on itself and on the number of encounters between these susceptibles and those who were infected one latent period ago. The number of infected individuals naturally increases by the same amount, and decreases with the number of people who recover. To consider an even more realistic model, we take into account the effect of vaccination as well, when the vaccinated individuals become recovered (immune) ones. \medskip

Moreover, one can consider the nonhomogeneous spatial distribution of the various populations as well. To do so we suppose that the infected individuals have a space-dependent influence on the susceptible ones. For instance, the healthy individuals get infected more likely closer to the infected ones. \medskip

We further consider a certain dynamics of the population, namely, the random movement (diffusion) of the individuals which leads to the faster transfer of the infection (see, e.g., \citet{He-Tsai}, \citet{Xu}). This process will be described by the Laplacian operator $\Delta:=\partial_x^2+\partial_y^2$ on $\Omega$ with the homogeneous Neumann boundary condition. \medskip 

Based on the considerations above, a compartment-type model can be formulated. Let $\beta>0$ denote the infection rate, $\gamma>0$ the recovery rate, and $\nu>0$ the vaccination rate. We also introduce a term $\mc I\colon[0,\infty)\to Y$ that will incorporate both the space-dependence of the infection process, and the time-delay involved.
More specifically, for all $t\in[0,\infty)$, we here let
\begin{equation}\label{eq:calI}
\mc I(t)=G\big(S(t-\delta),I(t-\delta),R(t-\delta)\big)
\end{equation}
for some appropriate function $G\colon X\to Y$. Typically, $G$ will depend only on the second coordinate, and for instance take the form of a convolution.

Then we consider the following system of (delayed) integro-differential equations:
\begin{equation}\label{eq:dxsir}
\left\{
\begin{aligned}
\tfrac{\dd}{\dd t} S(t) &= \Delta S(t)-\beta S(t)\mc I(t)-\nu S(t),  \\
\tfrac{\dd}{\dd t} I(t) &= \Delta I(t)+\beta S(t)\mc I(t)-\gamma I(t), \\
\tfrac{\dd}{\dd t} R(t) &= \Delta R(t)+\nu S(t)+\gamma I(t)
\end{aligned}
\right.
\end{equation}
for all $t\ge 0$. Due to the delay term $\mc{I}(t)$ we also need history functions $\varphi_S,\varphi_I,\varphi_R\colon[-\delta,0]\to Y$ such that
\begin{equation*}
S(s) = \varphi_S(s),\quad I(s) = \varphi_I(s),\quad R(s) = \varphi_R(s)
\end{equation*}
for all $s\in[-\delta,0]$. We assume that $\varphi_S(s),\varphi_I(s),\varphi_R(s)\geq 0$ holds for all $s\in[-\delta,0]$, and $S(s)+I(s)+R(s)$ is constant on $[-\delta,0]$.\medskip

Since the analytic solution to problem \eqref{eq:dxsir} is unknown, our aim is to approximate it using the Magnus-type integrator \eqref{eq:magnus}. To do so we introduce the function $u\colon[-\delta,\infty)\to X$ as
\begin{equation*}
u(t)=\big(S(t),I(t),R(t)\big)
\end{equation*}
for all $t\in[-\delta,\infty)$, and the operator family
\begin{equation*}
Q(w) =
\left(\begin{array}{lcr}
\Delta+\mc{M}_{-\beta G(w) }-\nu & 0 & 0 \\
\mc{M}_{\beta G(w) } & \Delta-\gamma & 0 \\
\nu & \gamma & \Delta
\end{array}\right)
\end{equation*}
for $w\in X$, where $\mc{M}_g\in\ms{L}(Y)$ denotes the multiplication operator $h\mapsto h\cdot g$ (this is a bounded operator whenever $g\in L^\infty(\Omega)$) and $\Delta$ is the Laplacian operator with the homogeneous Neumann boundary condition. Then the epidemic model \eqref{eq:dxsir} can be written as a quasilinear delay equation \eqref{eq:delay} with $F$ being the evaluation at $-\delta$. The latter means that this example is actually a special case of the point-delay presented in Example \ref{exm:point-delay} (with the choices of $\kappa_{\ell,\tau}$ and $F_{\ell,\tau}$ detailed in Remark \ref{rem:point-delay}), the Assumptions \eqref{ass:1}-\eqref{ass:3} are automatically satisfied. We further assume that $\varphi:=(\varphi_S,\varphi_I,\varphi_R)$ satisfies Assumption \eqref{ass:4}. \medskip

Our aim is to show that the Magnus-type integrator \eqref{eq:magnus} applied to the epidemic model \eqref{eq:dxsir} is convergent of second order. So by Theorem \ref{thm:main} we would need to check that Assumptions \eqref{Q}--\eqref{analytic} also hold, with $Q$ defined above.The issue is that we do not have uniform quasi-contractivity for $Q(w)$ if $w$ is allowed to run through all of $X$. Fortunately, there actually exists a natural invariant set $W\subset X$ that allows us to apply Theorem \ref{thm:main} and that prevents the blow-up of the solution. This invariant set will in addition ensure that the Magnus-type integrator preserves both positivity and the total population (i.e., the value $S+I+R$) when applied to problem \eqref{eq:dxsir}.

The next proposition covers Assumption \eqref{Q} under very mild conditions on the map $G$ (we assume neither continuity, nor linearity here).
\begin{proposition}\label{prop:invariant_example}
Let $\Omega\subset \RR^2$ be a bounded open set with boundary $\partial \Omega$ being a smooth Jordan curve, and $X:=(L^2(\Omega))^3$ the Banach lattice with norm $\|(x_1,x_2,x_3)\|_X^2:=\|x_1\|_2^2+\|x_2\|_2^2+\|x_3\|_2^2$. Further, let $\ms{I}>0$ be a constant, and
\begin{equation*}
W:=\Big\{w=(w_1,w_2,w_3)\in X_+\Big|
\int_\Omega w_1+w_2+w_3=\ms{I}
\Big\}.
\end{equation*}
Let further $G\colon X\to \C(\overline\Omega)$ be a positive
map such that there exists a constant $\mk{C}\geq0$ with
$\|G(w)\|_\infty\leq \mk{C} \|w\|_1$ for all $w\in W$.
Define the operators $\widetilde{P}(w)\in\ms{L}(X)$ for $w\in W$ as
\begin{equation}\label{eq:Qtilde}
\widetilde P(w) =
\left(\begin{array}{lcr}
\mc{M}_{-\beta G(w) }-\nu & 0 & 0 \\
\mc{M}_{\beta G(w) } & -\gamma & 0 \\
\nu & \gamma & 0
\end{array}\right),
\end{equation}
where $\beta,\gamma,\nu>0$ and for any $g\in \C(\overline\Omega)$, $\mc{M}_g\in\ms{L}(Y)$ denotes the multiplication operator $h\mapsto h\cdot g$. Let 
\begin{equation*}
H:=\left\{h\in Y\Big| h\in H^2(\Omega)\; \mbox{ and }\; \frac{\partial h}{\partial\mf{n}}=0 \mbox{ on }\Omega\right\},
\end{equation*}
and finally let $(P_0,D)$ be the diagonal Laplacian operator with the homogeneous Neumann boundary condition, i.e.,
\begin{equation*}
P_0 =
\left(\begin{array}{ccc}
\Delta & 0 & 0 \\
0 & \Delta & 0 \\
0 & 0 & \Delta
\end{array}\right)
\end{equation*}
with domain
\begin{equation*}
D =
\left\{
x\in X\Big|
x_j\in H \mbox{ for all } j=1,2,3
\right\}.
\end{equation*}
Then there exists $\alpha>0$ such that $\widetilde{P}(w)-\alpha\Id$ is dispersive (and in particular dissipative) for all $w\in W$, and the operators $(Q(w))_{w\in W}$ given by $Q(w):=P_0+\widetilde{P}(w)$ are generators of positive, uniformly quasi-contractive semigroups $(S_w(t))_{t\geq0}$ that leave $W$ invariant.
\end{proposition}
\begin{proof}
First let us show that for each $w\in W$ and $x\in X$ the (total population) function 
\begin{equation*}
t\mapsto\int_\Omega (S_w(t)x)_1+(S_w(t)x)_2+(S_w(t)x)_3
\end{equation*}
is constant. This, with positivity of the semigroups (to be shown after), would imply the invariance of $W$. Let us therefore consider some $x=(x_1,x_2,x_3)\in D$, and note that by differentiability of the orbit $t\mapsto S_w(t)x$ we have
\begin{align*}
&\ddt\int_\Omega (S_w(t)x)_1+(S_w(t)x)_2+(S_w(t)x)_3=\int_\Omega (Q(w)x)_1+(Q(w)x)_2+(Q(w)x)_3\\
=& \int_\Omega \left(\Delta x_1-\beta G(w)x_1-\nu x_1\right)+ \left(\Delta x_2+\beta G(w)x_1-\gamma x_2\right)+ \left(\Delta x_3+\nu x_1+\gamma x_2\right)\\
=& \int_\Omega \Delta x_1+\Delta x_2+\Delta x_3=\int_{\partial\Omega} \nabla\,x_1+\nabla\,x_2+\nabla\,x_3=0,
\end{align*}
where we used the divergence theorem and the boundary condition, so the integral remains constant whenever $x\in D$. Now $D$ is dense in $X$, so by standard arguments this holds for all $x\in X$. \medskip

Next, let us look at positivity. By classical PDE theory it is known that the Laplacian with Neumann boundary condition generates a strongly continuous contraction semigroup on $L^2(\Omega)$, and so $(P_0,D)$ also generates a strongly continuous contraction semigroup on $X$. We shall view each operator $P_0+\widetilde{P}(w)$ as a bounded perturbation of $P_0$.\\
By \citet[Thm. 13.3]{Batkai-etal} we have that if $A$ is the generator of a positive strongly continuous contraction semigroup, and $B$ is a dispersive and $A$-bounded operator with $A$-bound $a_0<1$, then $A+B$ is also the generator of a positive strongly continuous contraction semigroup. In this case we only aim for uniform quasi-contractivity, so it is enough to set $A=P_0$ and show that there exists some $\alpha>0$ such that each $\widetilde{P}(w)-\alpha\Id$ is dispersive (any bounded operator is $P_0$-bounded with bound 0). To show dispersivity, we resort to \citet[Prop. 11.12]{Batkai-etal}. Let us fix $w\in W$, $f\in X$, and a corresponding $f^*\in\ms{J}^+(f)$, i.e. an $f^*\in (X^*)_+=X_+$ such that $\langle f,f^*\rangle_X=\|f^+\|_X$ and $\|f^*\|_X\in\{0,1\}$. In this case $f^*$ is of the form $(f_1^+/\|f_1^+\|_Y,f_2^+/\|f_2^+\|_Y,f_3^+/\|f_3^+\|_Y)$ with the convention $0/0=0$.

We then have
\begin{align*}
&\langle (\widetilde{P}(w)-\alpha)f,f^*\rangle_X\\
=&\langle (-\beta G(w)f_1-\nu f_1,\beta G(w)f_1-\gamma f_2,\nu f_1+\gamma f_2), (f_1^+/\|f_1^+\|_Y,f_2^+/\|f_2^+\|_Y,f_3^+/\|f_3^+\|_Y)\rangle_X \\
&-\alpha\langle f,f^*\rangle \\
=& \nu\big(\langle f_1,f_3^+/\|f_3^+\|_Y\rangle-\langle f_1,f_1^+/\|f_1^+\|_Y\rangle\big)+\gamma\big(\langle f_2,f_3^+/\|f_3^+\|_Y\rangle-\langle f_2,f_2^+/\|f_2^+\|_Y\rangle\big) \\
&-\beta \langle G(w)f_1,f_1^+/\|f_1^+\|_Y\rangle+\beta \langle G(w)f_1,f_2^+/\|f_2^+\|_Y\rangle \\
&-\alpha \langle f_1,f_1^+/\|f_1^+\|_Y\rangle-\alpha \langle f_2,f_2^+/\|f_2^+\|_Y\rangle-\alpha \langle f_3,f_3^+/\|f_3^+\|_Y\rangle.
\end{align*}
Since $\langle f_j,f_j^+/\|f_j^+\|_Y\rangle=\|f_j^+\|_Y$ and $\langle f_j,f_k^+/\|f_k^+\|_Y\rangle\le\langle f_j^+,f_k^+/\|f_k^+\|_Y\rangle\le\|f_j^+\|_Y$ for all $j,k=1,2,3$, and due to the non-negativity of $G(w)$, we have
\begin{align*}
&\langle (\widetilde{P}(w)-\alpha)f,f^*\rangle_X \\
&\le \beta\langle G(w)f_1^+,f_2^+/\|f_2^+\|_Y\rangle-\beta\langle G(w)f_1^+,f_1^+/\|f_1^+\|_Y\rangle-\alpha\|f_1^+\|_Y-\alpha\|f_2^+\|_Y-\alpha\|f_3^+\|_Y \\
&\le(\beta\|G(w)\|_\infty-\alpha)\|f_1^+\|_Y\le (\beta C\|w\|_1-\alpha)\|f_1^+\|_Y=(\beta \mk{C}\ms I-\alpha)\|f_1^+\|_Y.
\end{align*}
Thus we may choose any $\alpha\ge\beta \mk{C}\ms{I}$ to obtain dispersive perturbations as required. Finally, note that a bounded dispersive operator is always dissipative.
\end{proof}

To be able to guarantee that Assumptions \eqref{smooth2}--\eqref{analytic} are also satisfied, we need to impose stronger conditions on $G$.

\begin{definition}
The linear function $G\colon X\to \C(\overline\Omega)$ is called \emph{well-adapted} if there exists a constant $\mk{C}>0$ such that the following estimates hold for all $w\in D$:
\begin{align}
\label{eq:Gw0} \|Gw\|_\infty &\le \mk{C}\|w\|_1, \\
\label{eq:Gw1} \|\partial_j(Gw)\|_\infty &\le \mk{C}\|w\|_1, \text{ for } j=1,2,3, \\
\label{eq:Gw2} \|\Delta(Gw)\|_\infty &\le \mk{C}\|w\|_1.
\end{align}
\end{definition}

Note that the Laplacian $P_0$ has 0 as an eigenvalue (with the constant functions in each coordinate as eigenvectors), so in order to define a norm on $D$, we have to shift the operator $P_0$.

\begin{definition}
Fix an $\veps>0$, set $Q_0:=P_0-\veps \mr{Id}$, and let $D$ be equipped with the norm
\[
\|x\|_D:=\|x\|_X+\|x\|_{Q_0}=\|x\|_X+\|Q_0x\|_X.
\]
It is well known that changing the value of $\veps>0$ leads to an equivalent norm.
Similarly, let $H$ be equipped with the norm
\[
\|f\|_H:=\|f\|_Y+\|f\|_{\Delta-\veps}=\|f\|_Y+\|(\Delta-\veps \mr{Id})f\|_Y.
\]
\end{definition}

\begin{lemma}\label{lem:D_C2}
Let $G\colon X\to\C(\overline\Omega)$ be well-adapted. Then the function $\widetilde P\colon D\to\Ell(D)$, defined in \eqref{eq:Qtilde}, is continuously differentiable.
\end{lemma}
\begin{proof}
It suffices to show that $\widetilde P$ is an affine map, that is, it is the sum of a bounded linear map from $D$ to $\Ell(D)$ and a constant. For $w\in D$ we have
\begin{equation*}
\widetilde P(w)=\left(\begin{array}{ccc}
\mc M_{-\beta(Gw)} -\nu &0 & 0 \\
\mc M_{\beta(Gw)} & -\gamma & 0 \\
\nu & \gamma & 0
\end{array}\right)
=\left(\begin{array}{ccc}
\mc M_{-\beta(Gw)} & 0 & 0 \\
\mc M_{\beta(Gw)} & 0 & 0 \\
0 & 0 & 0
\end{array}\right)
+\left(\begin{array}{ccc}
 -\nu & 0 & 0 \\
0 & -\gamma & 0 \\
\nu & \gamma & 0
\end{array}\right),
\end{equation*}
where $\mc M$ stands for the corresponding multiplication operator. Hence, we should show that the map $D\ni w\mapsto \mc M_{\beta(Gw)}\in\Ell(H)$ is bounded,
i.e., that $\|(Gw)f\|_H\le c_0\|w\|_D\cdot \|f\|_H$ for some appropriate constant $c_0>0$.  \medskip

To this end, for any $f\in Y$, we first rewrite the left-hand side by using the notation $g:=Gw$ as
\begin{equation}\label{eq:Gwfminden}
\begin{aligned}
\|gf\|_H&=\|gf\|_Y+\|\Delta(gf)-\veps gf)\|_Y \\
&=\|gf\|_Y+\|(\Delta g)f+\langle \nabla g,\nabla f\rangle+g(\Delta f)-\veps gf\|_Y \\
&\le \|gf\|_Y + \|(\Delta f-\veps f)g\|_Y+\|(\Delta g)f\|_Y+2\|\langle \nabla g,\nabla f\rangle\|_Y \\
&\le \|g\|_\infty\cdot\|f\|_Y+\|g\|_\infty\cdot\|f\|_{\Delta-\veps}+\|(\Delta g)f\|_Y+2\|\langle \nabla g,\nabla f\rangle\|_Y \\
&\le \|g\|_\infty\cdot\|f\|_H+\|\Delta g\|_\infty\cdot\|f\|_Y+2\|\langle \nabla g,\nabla f\rangle\|_Y.
\end{aligned}
\end{equation}
The norm of the scalar product can be estimated as follows
\begin{equation*}
\|\langle\nabla g,\nabla f\rangle\|_Y^2=\int_\Omega\big|\langle\nabla g,\nabla f\rangle\big|^2 \le \int_\Omega|\nabla g|^2\cdot|\nabla f|^2 \le 3\mk{C}^2\|w\|_1^2\cdot\int_\Omega|\nabla f|^2,
\end{equation*}
where we used inequality \eqref{eq:Gw1} from Lemma \ref{lem:Gw}. By using Green's identity and the homogeneous Neumann boundary condition in $H$, we rewrite the integral term as
\begin{align*}
\int_\Omega|\nabla f|^2&=\int_\Omega(\nabla f)(\nabla f) = \int_{\partial\Omega}f(\nabla f)\mf n-\int_\Omega f(\Delta f) = -\int_\Omega f(\Delta f) \\
&=-\int_\Omega f(\Delta f-\veps f)-\int_\Omega \veps f^2\le -\int_\Omega f(\Delta f-\veps f)\le \|f\|_Y\cdot\|(\Delta-\veps)f\|_Y \\
&\le \tfrac 14\big(\|f\|_Y+\|(\Delta-\veps)f\|_Y\big)^2,
\end{align*}
where we used the inequality of arithmetic and geometric means in the last step.
Altogether we have the inequality
\begin{equation*}\label{eq:nabla}
\begin{aligned}
\|\langle \nabla g,\nabla f\rangle\|_Y &\le \sqrt{3\mk{C}^2\|w\|_1^2\cdot\tfrac 14\big(\|f\|_Y+\|(\Delta-\veps)f\|_Y\big)^2} = \tfrac{{\sqrt3\mk{C}}}2\|w\|_1\cdot\big|\|f\|_Y+\|(\Delta-\veps)f\|_Y\big| \\
&=\tfrac{{\sqrt3\mk{C}}}2\|w\|_1\cdot\|f\|_H.
\end{aligned}
\end{equation*}
Since also $\|g\|_\infty\le \mk{C}\|w\|_1$ by inequality \eqref{eq:Gw0} in Lemma \ref{lem:Gw}, and $\|\Delta g\|_\infty\le \mk{C}\|w\|_1$ by \eqref{eq:Gw2}, the inequality \eqref{eq:Gwfminden} leads to
\begin{align*}
&\|(Gw)f\|_H\le\|g\|_\infty\cdot\|f\|_H+\|\Delta g\|_\infty\cdot\|f\|_Y+2\|\langle \nabla g,\nabla f\rangle\|_Y\le (2+\sqrt3)\mk{C}\|w\|_1\cdot\|f\|_H,
\end{align*}
hence
\begin{align*}
\big\|\mc M_{\beta(Gw)}\big\|_{\Ell(H)}=&\sup_{\|f\|_H\le 1}\|\beta(Gw)f\|_H\le \sup_{\|f\|_H\le 1} (2+\sqrt3)\mk{C}\beta\|w\|_1\cdot\|f\|_H \\
=& (2+\sqrt3)\mk{C}\beta\|w\|_1\le c_0\beta\|w\|_X \le c_0\beta\|w\|_D
\end{align*}
for some appropriate constant $c_0>0$.
Thus, the linear map $D\ni w\mapsto\mc M_{\beta(Gw)}\in\Ell(H)$ is bounded, and its bound depends on the infection rate $\beta$ and the constant $\mk{C}$ from Lemma \ref{lem:Gw}.
\end{proof}

\begin{remark}\label{rem:Gw}
In the above proof we actually showed $\big\|\mc M_{Gw}\big\|_{\Ell(H)}\leq (2+\sqrt3)\mk{C}\|w\|_1$, which for $w\in W\cap D$ means $\big\|\mc M_{Gw}\big\|_{\Ell(H)}\leq (2+\sqrt3)\mk{C}\ms{I}$.
\end{remark}

\begin{lemma}\label{lem:dissip}
Let $G\colon X\to \C(\overline\Omega)$ be a well-adapted function. Then there exists an $\eta>0$ such that the operator $\P(w)-\eta\Id$ is dissipative on $D$ for every $w\in W\cap D$.
\end{lemma}
\begin{proof}
We need to show that
\begin{equation}\label{eq:dissip}
\Re\left\langle\left(\begin{array}{ccc}\mc M_{-\beta(Gw)}-\nu-\eta &0&0\\ \mc M_{\beta(Gw)}&-\gamma-\eta&0\\ \nu&\gamma&-\eta\end{array}\right)
\left(\begin{array}{c}f_1\\f_2\\f_3\end{array}\right),\left(\begin{array}{c}f_1^*\\f_2^*\\f_3^*\end{array}\right)\right\rangle_{\sigma(D,D^*)}\le 0
\end{equation}
holds for all $f=(f_1,f_2,f_3)\in H$ and some elements $f_j^*\in\ms{J}(f_j)\subset H^*$ from the duality sets (i.e., $g^*\in\ms{J}(g)$ is an element of the dual space of norm 1 such that $\langle g,g^*\rangle=\|g\|_H$). Recall that $H$ is equipped with the norm $\|g\|_H:=\|g\|_Y+\|g\|_{\Delta-\veps}$. Before proceeding, we need to better understand the dual space $H^*$ and what we may choose as an element in $\ms{J}(g)$. Note that $H$ with the given norm is essentially the diagonal subspace of the elements of the form $(g,g)$ of the $\ell^1$-sum of the spaces $(Y,\|\cdot\|_Y)$ and $(H,\|\cdot\|_{\Delta-\veps})$. \medskip

Now $Y$ is a Hilbert space by definition, whilst $(H,\|\cdot\|_{\Delta-\veps})$ is a Hilbert space due to the norm being induced by the inner product $\langle h_1,h_2\rangle_{\Delta-\veps}:=\langle (\Delta-\veps)h_1,(\Delta-\veps)h_2\rangle_Y$.
Hence the dual $H^*$ is a factor space of the $\ell^\infty$ sum of the corresponding dual spaces. A suitable choice for $g^*\in\ms{J}(g)$ is then the element $(g/\|g\|_Y,g/\|g\|_{\Delta-\veps})\in H^*$ which acts on $(H,\|\cdot\|_H)$ as follows:
\begin{equation*}
\langle h,g^*\rangle_{\sigma(H,H^*)}:=\langle h,g/\|g\|_Y\rangle_Y +\langle h,g/\|g\|_{\Delta-\veps}\rangle_{\Delta-\veps}
\end{equation*}
and this is how we choose to define $f_j^*$ for $j=1,2,3$. \medskip

The weak evaluation in \eqref{eq:dissip} then expands to:
\begin{align*}
&\langle-\beta(Gw)f_1,f_1^*\rangle_{\sigma(H,H^*)}-(\nu+\eta)\langle f_1,f_1^*\rangle_{\sigma(H,H^*)}+\langle\beta(Gw)f_1,f_2^*\rangle_{\sigma(H,H^*)}\\&-(\gamma+\eta)\langle f_2,f_2^*\rangle_{\sigma(H,H^*)}+\nu\langle f_1,f_3^*\rangle_{\sigma(H,H^*)}
+\gamma\langle f_2,f_3^*\rangle_{\sigma(H,H^*)}-\eta\langle f_3,f_3^*\rangle_{\sigma(H,H^*)} \\
=&-\eta\big(\|f_1\|_H+\|f_2\|_H+\|f_3\|_H\big)+\Big\langle-\beta(Gw)f_1,\frac{f_1}{\|f_1\|_Y}\Big\rangle_Y+\Big\langle-\beta(Gw)f_1,\frac{f_1}{\|f_1\|_{\Delta-\veps}}\Big\rangle_{\Delta-\veps} \\
&-\nu\|f_1\|_H+\Big\langle\beta(Gw)f_1,\frac{f_2}{\|f_2\|_Y}\Big\rangle_Y+\Big\langle\beta(Gw)f_1,\frac{f_2}{\|f_2\|_{\Delta-\veps}}\Big\rangle_{\Delta-\veps}-\gamma\|f_2\|_H \\
&+\nu\Big\langle f_1,\frac{f_3}{\|f_3\|_Y}\Big\rangle_Y+\nu\Big\langle f_1,\frac{f_3}{\|f_3\|_{\Delta-\veps}}\Big\rangle_{\Delta-\veps}+\gamma\Big\langle f_2,\frac{f_3}{\|f_3\|_Y}\Big\rangle_Y+\gamma\Big\langle f_2,\frac{f_3}{\|f_3\|_{\Delta-\veps}}\Big\rangle_{\Delta-\veps}=:(\ast).
\end{align*}
The properties of the dual elements, Remark \ref{rem:Gw} and the bound \eqref{eq:Gw0} in Lemma \ref{lem:Gw} imply the following inequalities for $j,k=1,2,3$:
\begin{align*}
&\left|\Re\Big\langle\beta(Gw)f_j,\frac{f_k}{\|f_k\|_Y}\Big\rangle_Y \right|\le \left|\Big\langle\beta(Gw)f_j,\frac{f_k}{\|f_k\|_Y}\Big\rangle_Y \right|
\le|\beta(Gw)f_j\|_Y\cdot\Big\|\frac{f_k}{\|f_k\|_Y}\Big\|_Y\\& \le \beta\|\mc M_{Gw}\|_{\Ell(Y)}\cdot\|f_j\|_Y 
\le \beta\mk{C}\|w\|_1\cdot\|f_j\|_Y = \beta\mk{C}\ms{I}\|f_j\|_Y ,
\end{align*}
\begin{align*}
&\left|\Re\Big\langle\beta(Gw)f_j,\frac{f_k}{\|f_k\|_{\Delta-\veps}}\Big\rangle_{\Delta-\veps}\right|
\le \left|\Big\langle\beta(Gw)f_j,\frac{f_k}{\|f_k\|_{\Delta-\veps}}\Big\rangle_{\Delta-\veps}\right|\\ &\le \|\beta(Gw)f_j\|_{\Delta-\veps}\cdot\Big\|\frac{f_k}{\|f_k\|_{\Delta-\veps}}\Big\|_{\Delta-\veps} \le \beta\|(Gw)f_j\|_H\\
&\le \beta\|\mc M_{Gw}\|_{\Ell(H)}\cdot\|f_j\|_H\le (2+\sqrt3)\beta\mk{C}\ms{I}\|f_j\|_H,
\end{align*}
\begin{align*}
&\Re\Big\langle f_j,\frac{f_k}{\|f_k\|_Y}\Big\rangle+\Re\Big\langle f_j,\frac{f_k}{\|f_k\|_{\Delta-\veps}}\Big\rangle_{\Delta-\veps}
\le \|f_j\|_Y\cdot\Big\|\frac{f_k}{\|f_k\|_Y}\Big\|_Y + \|f_j\|_{\Delta-\veps}\cdot\Big\|\frac{f_k}{\|f_k\|_{\Delta-\veps}}\Big\|_{\Delta-\veps} \\
&=\|f_j\|_Y+\|f_j\|_{\Delta-\veps}=\|f_j\|_H.
\end{align*}
Then we obtain
\begin{align*}
\Re\,(\ast) \le& -\eta\big(\|f_1\|_H+\|f_2\|_H+\|f_3\|_H\big)+(6+2\sqrt3)\beta\mk{C}\ms{I}\|f_1\|_H-\nu\|f_1\|_H \\
&-\gamma\|f_2\|_H+\nu\|f_1\|_H+\gamma\|f_2\|_H\\
=&-\eta\big(\|f_1\|_H+\|f_2\|_H+\|f_3\|_H\big)+(6+2\sqrt3)\beta\mk{C}\ms{I}\|f_1\|_H.
\end{align*}
That is, one may choose $\eta=(6+2\sqrt3)\beta\mk{C}\ms{I}$.
\end{proof}

\begin{lemma}\label{lem:dissip2}
For any $\alpha<\pi/2$ there exists an $\eta>0$ such that the operator $\ee^{\ii\phi}(\P(w)-\eta\Id)$ is dissipative on $X$ for every $w\in W$ and $\phi\leq\alpha$.
\end{lemma}
\begin{proof}
By \citet[Prop.~II.3.23]{Engel-Nagel}, we need to show that
\begin{equation}\label{eq:dissipexp}
\Re\left\langle\left(\begin{array}{ccc}\ee^{\ii\phi}(\mc M_{-\beta(Gw)}-\nu-\eta) &0&0\\ \ee^{\ii\phi}\mc M_{\beta(Gw)}&\ee^{\ii\phi}(-\gamma-\eta)&0\\ \ee^{\ii\phi}\nu&\ee^{\ii\phi}\gamma&-\ee^{\ii\phi}\eta\end{array}\right)
\left(\begin{array}{c}f_1\\f_2\\f_3\end{array}\right),\left(\begin{array}{c}f_1^\star\\f_2^\star\\f_3^\star\end{array}\right)\right\rangle_X\le 0
\end{equation}
holds for all $f=(f_1,f_2,f_3)\in H$ and $f_j^\star=f_j/\|f_j\|_Y$ for $j=1,2,3$ (again, $0/0=0$). Similarly as in the proof of Lemma \ref{lem:dissip}, we have the following upper bound on the real part of the scalar product:
\begin{align*}
&\Re\ee^{\ii\phi}\Big\langle -\beta(Gw)f_1,\frac{f_1}{\|f_1\|_Y}\Big\rangle_Y - \Re\ee^{\ii\phi}(\nu+\eta)\|f_1\|_Y + \Re\ee^{\ii\phi}\Big\langle -\beta(Gw)f_1,\frac{f_2}{\|f_2\|_Y}\Big\rangle_Y \\
&- \Re\ee^{\ii\phi}(\gamma+\eta)\|f_2\|_Y + \Re\ee^{\ii\phi}\nu\Big\langle f_1,\frac{f_3}{\|f_3\|_Y}\Big\rangle_Y
+ \Re\ee^{\ii\phi}\gamma\Big\langle f_1,\frac{f_2}{\|f_3\|_Y}\Big\rangle_Y - \Re\ee^{\ii\phi}\eta\|f_3\|_Y \\
&\le \|f_1\|_Y\big(2\beta\|\mc M_{Gw}\|_{\Ell(Y)}+\nu-\Re\ee^{\ii\phi}(\nu+\eta)\big) + \|f_2\|_Y\big(\gamma-\Re\ee^{\ii\phi}(\gamma+\eta)\big) - \|f_3\|_Y\Re\ee^{\ii\phi}\eta\\
&\le \|f_1\|_Y\big(2\beta\|Gw\|_{\infty}+\nu-\Re\ee^{\ii\phi}(\nu+\eta)\big) + \|f_2\|_Y\big(\gamma-\Re\ee^{\ii\phi}(\gamma+\eta)\big) - \|f_3\|_Y\Re\ee^{\ii\phi}\eta
.
\end{align*}
By \eqref{eq:Gw0} we have $\|Gw\|_{\infty}\leq\mk{C} \|w\|_1=\mk{C}\ms{I}$, and also $0<\cos\alpha\leq\cos\phi$.
Thus one may choose 
\[
\eta=\max\left\{\frac{2\beta\mk{C}\ms{I}+(1-\cos\alpha)\nu}{\cos \alpha}
,\frac{(1-\cos\alpha)\gamma}{\cos\alpha}
\right\}
\]
to satisfy inequality \eqref{eq:dissipexp}.
\end{proof}

In applications, a widely used class of maps $G\colon X\to\C(\overline\Omega)$ is that of convolution-type operators \citep[as in][]{Takacs-etal}, i.e., we have
\begin{equation}\label{eq:Gconvolution}
(Gw)(z):=\int_\Omega \langle w(x),h(z-x)\rangle\dd x
\end{equation}
for all $w\in X$ and $z\in\Omega$, where $h\in\left(\C^2\left(\overline{\Omega-\Omega}\right)\right)^3$ is the convolution kernel and $\langle\cdot,\cdot\rangle$ denotes the scalar product in $\RR^3$. Then the terms $\pm\beta S(t)\mc I(t)$ in \eqref{eq:dxsir} describe an infection process where the point-to-point infection rate depends only on the directed vector between the points. We note that in most cases, $G$ depends only on the distribution $I$ of the infected individuals, and not on $S$ or $R$. \medskip

\begin{lemma}\label{lem:Gw}
The map $G\colon X\to\C(\overline\Omega)$ defined in \eqref{eq:Gconvolution} is well-adapted.
\end{lemma}
\begin{proof}
The convolution \eqref{eq:Gconvolution} is linear in $w$. It remains to show the validity of the norm estimates.
The inequality \eqref{eq:Gw0} directly follows from the convolution form \eqref{eq:Gconvolution} as
\begin{align*}
\|Gw\|_\infty&=\sup_{z\in\Omega}|(Gw)(z)| = \sup_{z\in\Omega}\Big|\int_\Omega\langle w(x),h(z-x)\rangle \dd x\Big|\le\sup_{z\in\Omega}\int_\Omega\big|\langle w(x),h(z-x)\rangle\big|\dd x \\
&=\sup_{z\in\Omega}\int_\Omega|w(x)|\cdot|h(z-x)|\dd x\le \|h\|_\infty\cdot\int_\Omega |w(x)|\dd x=\|h\|_\infty\cdot\|w\|_1.
\end{align*}
The inequality \eqref{eq:Gw1} can be shown similarly for $j=1,2,3$:
\begin{align*}
\|\partial_j(Gw)\|_\infty &= \sup_{z\in\Omega}\big|\big(\partial_j(Gw)\big)(z)\big|=\sup_{z\in\Omega}\Big|\partial_j\int_\Omega\langle w(x),h(z-x)\rangle\dd x\Big| \\
&\le \sup_{z\in\Omega}\int_\Omega\big|\partial_j\langle w(x),h(z-x)\rangle\big|\dd x = \sup_{z\in\Omega}\int_\Omega\big|\langle w(x),(\partial_jh)(z-x)\rangle\big|\dd x \\
&\le \sup_{z\in\Omega}\int_\Omega|w(x)|\cdot|(\partial_jh)(z-x)|\dd x \le \|\partial_jh\|_\infty\cdot\int_\Omega|w(x)|\dd x=\|\partial_jh\|_\infty\cdot\|w\|_1.
\end{align*}
To prove the last inequality \eqref{eq:Gw2}, we write
\begin{align*}
\|\Delta(Gw)\|_\infty &= \sup_{z\in\Omega}\big|\big(\Delta(Gw)\big)(z)\big|=\sup_{z\in\Omega}\Big|\int_\Omega\langle w(x),h(z-x)\rangle\dd x\Big| \\
&\le \sup_{z\in\Omega}\int_\Omega \big|\Delta\langle w(x),h(z-x)\rangle\big|\dd x = \sup_{z\in\Omega}\int_\Omega \big|\langle w(x),(\Delta h)(z-x)\rangle\big|\dd x \\
&\le \sup_{z\in\Omega}\int_\Omega |w(x)|\cdot|(\Delta h)(z-x)|\dd x \le \|\Delta h\|_\infty\cdot\int_\Omega|w(x)|\dd x
= \|\Delta h\|_\infty \cdot\|w\|_1
\end{align*}
with the notation $\Delta h = (\Delta h_1,\Delta h_2,\Delta h_3)$.
\end{proof}

\begin{proposition}\label{prop:example_D}
Assume that the conditions of Proposition \ref{prop:invariant_example} are satisfied, and $G$ is of the convolution form given in \eqref{eq:Gconvolution}. Then Assumptions \eqref{smooth2}--\eqref{analytic} also hold for $Q_0$, $\Q:=\P+\veps\mr{Id}$, $\alpha\in(0,\pi/2)$ and an appropriate $c>0$.
\end{proposition}
\begin{proof}
First let us show Assumption \eqref{smooth2}. Note that the map $w\mapsto \mc{M}_{\beta(Gw)}$ is linear and bounded as a map $X\to \C(\overline\Omega)$, and so $w\mapsto\widetilde{Q}(w)$ is a bounded affine map $X\to \ms{L}(X)$, hence infinitely differentiable.\\
Let $c:=\max\left\{\frac{2\beta\mk{C}\ms{I}+(1-\cos\alpha)\nu}{\cos \alpha}
,\frac{(1-\cos\alpha)\gamma}{\cos\alpha},(6+2\sqrt3)\beta\mk{C}\ms{I}\right\}+\veps$.
Then Assumption \eqref{dissip} follows from Lemmas \ref{lem:D_C2} and \ref{lem:dissip}, whilst Assumption \eqref{analytic} follows from Lemma \ref{lem:dissip2} and the standard fact that the homogeneous Neumann Laplacian is analytic on $L^2(\Omega)$ with angle $\pi/2$.
\end{proof}

Combining Propositions \ref{prop:invariant_example} and \ref{prop:example_D}, and recalling that Assumptions \eqref{ass:1}--\eqref{ass:3} are automatically satisfied with the choices detailed in Remark \ref{rem:point-delay}, we see that Theorem \ref{thm:main} can be applied to our example, and we obtain the main result of this section.

\begin{corollary}
Assume that the conditions of Proposition \ref{prop:invariant_example} are satisfied, and $G$ is of the convolution form given in \eqref{eq:Gconvolution}. Further, let the initial history function $\varphi$ satisfy Assumption \eqref{ass:4}. Then we have convergence of second order on any compact time-interval of the Magnus-type integrator \eqref{eq:magnus} applied to the epidemic model \eqref{eq:dxsir}. Moreover, the total population remains constant and the positivity of the solution is preserved.
\end{corollary}

\begin{remark}
We may also combine the above setting with Example \ref{exm:integral}, reflecting a different transmission/infection dynamic. Instead of formula \eqref{eq:calI}, we then consider
\begin{equation}\label{eq:calIintegral}
\mc I(t):=\frac 2\delta\int_{-\delta}^{-\delta/2} G\left(u(t+s)\right)\dd s
\end{equation}
in the model \eqref{eq:dxsir}, with the same convolution $G$ as defined before in \eqref{eq:Gconvolution}. This would correspond to an infection process where the latent period is not fixed of time $\delta$, but rather exhibits a uniform distribution within a timeframe between $\delta/2$ and $\delta$.\\
Actually, other distributions for the latent period would lead to a further modified version of formula \eqref{eq:calIintegral} of the form
\begin{equation}\label{eq:calIintegral2}
\mc I(t):=\mb{E}_{\mb{S}} G(u(t-\mb{S})),
\end{equation}
where $\mb{S}\in[\epsilon,\delta]$ is the random variable that encodes the latent period (we assume that the distribution of the latent period remains constant over time). Since the convolution operator is a continuous linear map, it commutes with the expected value, and so
\[
\mc I(t):= G(\mb{E}_{\mb{S}}u(t-\mb{S})).
\]
Now recall that our choice for $W$ was convex, and so invariant under the expectation operator, so the above arguments apply, and Theorem \ref{thm:main} remains applicable for this wide range of models as well.\\
\end{remark}


\bigskip
\paragraph{Acknowledgement} Both authors acknowledge the Bolyai J{\'a}nos Research Scholarship of the Hungarian Academy of Sciences. P.~Cs.~was also supported by the Ministry of Innovation and Technology NRDI Office under the grant SNN-125119 and within the framework of the Artificial Intelligence National Laboratory Program, Budapest. D.~K.-K. was also supported by the NRDI ``{\'E}lvonal'' KKP-133921 grant.



\begin{thebibliography}{12}

\bibitem[B{\'a}tkai \emph{et al.}(2017) B{\'a}tkai, Kramar Fijav\v{z} \& Rhandi]{Batkai-etal}
\textsc{B{\'a}tkai, A., Kramar Fijav\v{z}, M. \& Rhandi. A.} (2017)
\newblock \emph{Positive Operator Semigroups, From Finite to Infinite Dimensions}.
\newblock Cham: Birkh\"auser.

\bibitem[B{\'a}tkai \& Sikolya(2012)B{\'a}tkai \& Sikolya]{Batkai-Sikolya}
\textsc{B{\'a}tkai, A. \& Sikolya, E.} (2012)
\newblock The norm convergence of a Magnus expansion method.
\newblock \emph{Cent.~Eur.~J.~Math.}, \textbf{10}, 150--158.

\bibitem[Blanes \emph{et al.}(1998)Blanes, Casas, Oteo \& Ros]{Blanes-etal1}
\textsc{Blanes, S., Casas, F., Oteo, J. A. \& Ros, J.} (1998)
\newblock Magnus and Fer expansions for matrix differential equations: the convergence problem.
\newblock \emph{J.~Phys.~A: Math.~Gen.}, \textbf{31}, 259--268.

\bibitem[Blanes \emph{et al.}(2009)Blanes, Casas, Oteo \& Ros]{Blanes-etal2}
\textsc{Blanes, S., Casas, F., Oteo, J. A. \& Ros, J.} (2009)
\newblock The Magnus expansion and some of its applications.
\\newblock emph{Phys.~Rep.}, \textbf{470}, 151--238.

\bibitem[Casas \& Iserles(2006)Casas \& Iserles]{Casas-Iserles}
\textsc{Casas, F. \& Iserles, A.} (2006)
\newblock Explicit Magnus expansions for nonlinear equations.
\newblock \emph{J.~Phys.~A: Math.~Gen.}, \textbf{39}, 5445--5462.

\bibitem[Csom{\'o}s(2020)Csom{\'o}s]{Csomos}
\textsc{Csom{\'o}s, P.} (2020)
\newblock Magnus-type integrator for semilinear delay equations with an application to epidemic models.
\newblock \emph{J.~Comput.~Appl.~Math.}, \textbf{363}, 92--105.

\bibitem[Csom{\'o}s \& Tak{\'a}cs(2021)Csom{\'o}s \& Tak{\'a}cs]{Csomos-Takacs}
\textsc{Csom{\'o}s, P. \& Tak{\'a}cs, B.} (2021)
\newblock Operator splitting for space-dependent epidemic model.
\newblock \emph{Appl. Numer. Math.}, \textbf{159}, 259--280.

\bibitem[Engel \& Nagel(2000)Engel \& Nagel]{Engel-Nagel}
\textsc{Engel, K. J. \& Nagel, R.} (2000)
\newblock \emph{One-Parameter Semigroups for Linear Evolution Equations},
\newblock New York: Springer-Verlag.

\bibitem[Gonz{\'a}lez \emph{et al.}(2006)Gonz{\'a}lez, Ostermann \& Thalhammer]{Gonzalez-etal}
\textsc{Gonz{\'a}lez, C., Ostermann, A. \& Thalhammer, M.} (2006)
\newblock A second-order Magnus-type integrator for nonautonomous parabolic problems.
\newblock \emph{J.~Comput.~Appl.~Math.}, \textbf{189}, 142--156.

\bibitem[He \& Tsai(2019)He \& Tsai]{He-Tsai}
\textsc{He, J. \& Tsai, J. C.} (2019)
\newblock Traveling waves in the Kermack--McKendrick epidemic model with latent period.
\newblock \emph{Z. Angew. Math. Phys.}, \textbf{70}, 27.

\bibitem[Huang \& Takeuchi(2011)Huang \& Takeuchi]{Huang-Takeuchi}
\textsc{Huang, G. \& Takeuchi, Y.} (2011)
\newblock Global analysis on delay epidemiological dynamic models with nonlinear incidence.
\newblock \emph{J.~Math.~Biol.}, \textbf{63}, 125--139.

\bibitem[Kato(1953)Kato]{Kato53} 
\textsc{Kato, T.} (1953)
\newblock Integration of the equation of evolution in a Banach space.
\newblock \emph{J. Math. Soc. Jpn.}, \textbf{5}, 208--234.

\bibitem[Kendall(1965)Kendall]{Kendall}
\textsc{Kendall, D. G.} (1965)
\newblock Mathematical models of the spread of infection.
\newblock In: \emph{Mathematics and Computer Science in Biology and Medicine}. London: H.M.S.O., 213--225.  

\bibitem[Kermack \& McKendrick(1927)Kermack \& McKendrick]{Kermack-McKendrick-1927}
\textsc{Kermack, W. \& McKendrick, A.} (1927)
\newblock A contribution to the mathematical theory of epidemics.
\newblock \emph{Proc.~Roy.~Soc.~Lond.~A}, \textbf{115}, 700--721.

\bibitem[Magnus(1954)Magnus]{Magnus}
\textsc{Magnus, W.} (1954)
\newblock On the exponential solution of a differential equation for a linear operator.
\newblock \emph{Comm.~Pure Appl~Math.}, \textbf{7}, 649--673.

\bibitem[Moan \& Niesen(2008)Moan \& Niesen]{Moan-Niesen}
\textsc{Moan, P. C. \& Niesen, J.} (2008)
\newblock Convergence of the Magnus series.
\newblock \emph{J.~Found.~Comput.~Math.}, \textbf{8}, 291--301.

\bibitem[Nickel(1997)Nickel]{Nickel}
\textsc{Nickel, G.} (1997)
\newblock Evolution semigroups for nonautonomous Cauchy problems.
\newblock \emph{Abstr.~Appl.~Anal.}, \textbf{2}, 73--95.

\bibitem[Tak{\'a}cs \emph{et al.}(2020)Tak{\'a}cs, Horv{\'a}th \& Farag{\'o}]{Takacs-etal}
\textsc{Tak{\'a}cs, B., Horv{\'a}th, R. \& Farag{\'o}, I} (2020)
\newblock Space dependent models for studying the spread of some diseases.
\newblock \emph{Comp. Math. Appl.}, \textbf{80} 395--404.

\bibitem[Xu(2014)Xu]{Xu}
\textsc{Xu, Z.} (2014)
\newblock Traveling waves in a Kermack--McKendrick epidemic model with diffusion and latent period.
\newblock \emph{Nonlin. Anal.}, \textbf{111}, 66--81.

\end{thebibliography}
\end{document}